\documentclass[a4paper]{amsart}
\usepackage{amsmath}
\usepackage{amsfonts}
\usepackage{amsthm}
\usepackage{amssymb}
\usepackage{enumerate}

\newcommand{\C}{\mathbb{C}}
\newcommand{\Q}{\mathbb{Q}}
\newcommand{\R}{\mathbb{R}}
\newcommand{\F}{\mathbb{F}}
\newcommand{\Z}{\mathbb{Z}}
\renewcommand{\P}{\mathbb{P}}
\DeclareMathOperator{\Aut}{Aut}
\DeclareMathOperator{\Pic}{Pic}
\DeclareMathOperator{\Gal}{Gal}
\DeclareMathOperator{\im}{im}

\DeclareMathOperator{\rk}{rk}

\DeclareMathOperator{\Nef}{Nef}
\DeclareMathOperator{\Amp}{Amp}
\DeclareMathOperator{\Br}{Br}
\DeclareMathOperator{\Spec}{Spec}
\newcommand{\pos}{{\mathcal{C}}}

\newcommand{\Picsch}{\mathbf{Pic}}
\newcommand{\Autsch}{\mathbf{Aut}}
\newcommand{\et}{\textrm{{\'e}t}}
\newcommand{\sigmaf}[2]{{{}^{#1} #2}}

\newtheorem{theorem}{Theorem}[section]
\newtheorem{proposition}[theorem]{Proposition}
\newtheorem{lemma}[theorem]{Lemma}
\newtheorem{corollary}[theorem]{Corollary}
\theoremstyle{definition}
\newtheorem{defn}[theorem]{Definition}
\newtheorem{example}[theorem]{Example}
\theoremstyle{remark}
\newtheorem{remark}[theorem]{Remark}

\newcommand{\Gk}{{\Gamma_k}}
\newcommand{\kbar}{\bar{k}}
\newcommand{\ksep}{k^s}
\newcommand{\Xbar}{{\bar{X}}}
\newcommand{\Xsep}{X^s}
\renewcommand{\H}{\mathrm{H}}

\newcommand{\OO}{{\mathcal{O}}}
\renewcommand{\O}{{\mathrm{O}}}

\renewcommand{\L}{\Lambda}
\renewcommand{\S}{{\mathcal{S}}}
\title{Finiteness results for K3 surfaces over arbitrary fields}
\author{Martin Bright}
\address{Mathematisch Instituut \\ Niels Bohrweg 1 \\ 2333 CA Leiden \\ Netherlands}
\email{m.j.bright@math.leidenuniv.nl}
\author{Adam Logan}
\address{The Tutte Institute for Mathematics and Computation,
P.O. Box 9703, Terminal, Ottawa, ON K1G 3Z4, Canada}
\address{School of Mathematics and Statistics, 4302 Herzberg Laboratories,
  1125 Colonel By Drive, Carleton University, Ottawa, ON K1S 5B6, Canada}
\email{adam.m.logan@gmail.com}
\author{Ronald van Luijk}
\address{Mathematisch Instituut \\ Niels Bohrweg 1 \\ 2333 CA Leiden \\ Netherlands}
\email{rvl@math.leidenuniv.nl}


\begin{document}

\begin{abstract}
Over an algebraically closed field, various finiteness results are known regarding the automorphism group of a K3 surface and the action of the automorphisms on the Picard lattice.  We formulate and prove versions of these results over arbitrary base fields, and give examples illustrating how behaviour can differ from the algebraically closed case.
\end{abstract}

\maketitle

\section{Introduction}

The geometry of K3 surfaces over the complex numbers has a long history, with many results known about the cohomology, the Picard group, and the automorphism group of an algebraic K3 surface, and how these objects interact.
Such results over the complex numbers carry over to other algebraically closed fields of characteristic zero, and similar results are also known over algebraically closed fields of other characteristics.
For a comprehensive treatment of the geometry of K3 surfaces, we refer the reader to the lecture notes of Huybrechts~\cite{huybrechts}.
Much of the theory we use was originally developed by Nikulin~\cite{nikulin80}.

K3 surfaces are also interesting from an arithmetic point of view, with much recent work on understanding the rational points, curves, Brauer groups and other invariants of K3 surfaces over number fields.
In this article, we investigate the extent to which some standard finiteness results for K3 surfaces over algebraically closed fields remain true over more general base fields.
In particular, we show how to define the correct analogue of the Weyl group, and give an explicit description of it.
This allows us to formulate and prove finiteness theorems over arbitrary fields, modelled on those already known over algebraically closed fields.
The tools we use include representability of the Picard and automorphism schemes, classification of transitive group actions on Coxeter--Dynkin diagrams, and an explicit description of the walls of the ample cone. 
We follow these theoretical results with several detailed examples, showing how the relationship between the Picard group and the automorphism group can be different from the geometric case.
We end by proving that a surface over $\Q$ of the form $x^4-y^4 = c(z^4-w^4)$
has finite automorphism group for $c \in \Q^*$ that is not in the subgroup
generated by squares together with $-1, 2$.

The specific finiteness results we address go back to Sterk~\cite{sterk}.  To state them, we need some definitions; we follow the notation of~\cite{huybrechts}.
In this article, by a K3 surface we will always mean an algebraic K3 surface, which is therefore projective.

Let $k$ be a field, and let $X$ be a K3 surface over $k$.
Denote the group of isometries of $\Pic X$ by $\O(\Pic X)$.
Reflection in any $(-2)$-class in $\Pic X$ defines an isometry of $\Pic X$, and we define the \emph{Weyl group} $W(\Pic X) \subset \O(\Pic X)$ to be the subgroup generated by these reflections.

We recall the definitions of the positive, ample and nef cones associated to the K3 surface $X$; see also~\cite[Chapter~8]{huybrechts}.  Let $(\Pic X)_\R$ denote the real vector space $(\Pic X) \otimes_\Z \R$.
By the Hodge index theorem, the intersection product on $(\Pic X)_\R$ has signature $(1, \rho-1)$; so the set $\{ \alpha \in (\Pic X)_\R \mid \alpha^2 > 0 \}$ consists of two connected components.
The \emph{positive cone} $\pos_X \subset (\Pic X)_\R$ is the connected component containing all the ample classes.

The \emph{ample cone} $\Amp(X)$ is the cone in $(\Pic X)_\R$ generated by all classes of ample line bundles.  The \emph{nef cone} $\Nef(X)$ is defined as
\[
\Nef(X) = \{ \alpha \in (\Pic X)_\R \mid \alpha \cdot C \ge 0 \text{ for all curves } C \subset X \}.
\]
Finally, we define $\Nef^e(X)$ to be the real convex hull of $\Nef(X) \cap \Pic X$.
An application of the criterion of Nakai--Moishezon--Kleiman 
shows that $\Amp(X)$ is the interior of $\Nef(X)$ and $\Nef(X)$ is the closure of $\Amp(X)$: see~\cite[Corollary~8.1.4]{huybrechts}.

The following finiteness theorems are due to Sterk~\cite{sterk} for $k=\C$ and to Lieblich and Maulik~\cite{lm} when $k$ has positive characteristic not equal to $2$.
As in Huybrechts~\cite[Chapter~8]{huybrechts}, a \emph{fundamental domain} for the action of a discrete group $G$ acting continuously on a topological manifold $M$ is defined as the closure $\overline{U}$ of an open subset $U\subset M$ such that $M = \cup_{g \in G} g \overline{U}$
and such that for $g \neq h \in G$ the intersection $g\overline{U} \cap h\overline{U}$ does not contain interior points of $gU$ or $hU$.
\begin{theorem}\label{thm:old}
Let $k = \kbar$ be an algebraically closed field of characteristic not equal to~$2$, and let $X$ be a K3 surface over $k$.
\begin{enumerate}
\item \cite[Corollary~8.2.11]{huybrechts} The cone $\Nef(X) \cap \pos_X$ is a fundamental domain for the action of $W(\Pic X) \subset \O(\Pic X)$ on the positive cone $\pos_X$.
\item\label{it2} \cite[Theorem~15.2.6]{huybrechts}, \cite[Proposition~5.2]{lm} The subgroup $W(\Pic X)$ is normal in $\O(\Pic X)$; the natural map $\Aut X \to \O(\Pic X)/W(\Pic X)$ has finite kernel and image of finite index.
\item \cite[Theorem~8.4.2]{huybrechts} The action of $\Aut X$ on $\Nef^e X$ admits a rational polyhedral fundamental domain.
\item \cite[Corollary~8.4.6]{huybrechts} The set of orbits under $\Aut X$ of $(-2)$-curves on $X$ is finite.  More generally, for any $d$ there are only finitely many orbits under $\Aut X$ of classes of irreducible curves of self-intersection $2d$.
\end{enumerate}
\end{theorem}
In Section~\ref{sec:ft} we will prove analogues of the various statements of Theorem~\ref{thm:old} when $k$ is replaced by an arbitrary base field of characteristic different from $2$. 

A consequence of Theorem~\ref{thm:old}(\ref{it2}) is that the finiteness of $\Aut X$ depends only on the lattice $\Pic X$.  Those possible Picard lattices for which $\O(\Pic X)/W(\Pic X)$ is finite have been classified~\cite{nikulin}.
Over an arbitrary base field we will see that, instead of using the Weyl group $W(\Pic X)$, we must use the Galois-invariant part of the geometric Weyl group.
This means that the finiteness of $\Aut X$ is no longer determined purely by the Picard lattice $\Pic X$; rather, it depends on the geometric Picard lattice together with the Galois action.
In Section~\ref{sec:eg}, we give several examples that illustrate this difference to the classical case.

We thank the referee, whose thorough report helped us improve the exposition and correct a number of errors in the paper.
The second author would like to thank the Tutte Institute for Mathematics and Computation for its partial support for a visit to the University of Leiden during which much of this research was done.

\section{Lemmas on lattices}
In this section we will study lattices with the action of a 
group.  Given a lattice $\L$ with the action of a finite group $H$, we
consider the group of automorphisms that commute with $H$, and the
group of automorphisms that preserve the sublattice fixed by $H$.

\begin{defn} 
A \emph{lattice} $\L$ is a free abelian group of finite rank with
a nondegenerate integer-valued symmetric bilinear form.  If the form is
(positive or negative) definite, we likewise refer to $\L$ as 
\emph{definite}.  The group of automorphisms of $\L$ preserving the form
is denoted $\O(\L)$.  A \emph{sublattice} of $\L$ is a subgroup on
which the restriction of the form is nondegenerate.  Given a
sublattice $M \subseteq \L$, we use $\O(\L,M)$ for the subgroup of
$\O(\L)$ fixing $M$ as a set.  For a subgroup $H \subseteq \O(\L)$,
let $\L^H$ be the subgroup of $\L$ consisting of the elements fixed
by every element of $H$
(note that according to our conventions $\L^H$ may not be a lattice,
because the quadratic form on $\L$ may be degenerate when restricted to $\L^H$).
The vector space $\L \otimes_\Z \Q$ will be denoted $\Lambda_\Q$.
\end{defn}

Our goal is to prove the following proposition.

\begin{proposition}\label{lattice-prop}
Let $\L$ be a lattice and $H \subseteq \O(\L)$ a subgroup such that $M = \L^H$ is a
lattice.  Then the following hold:
\begin{enumerate}
\item the natural map $\O(\L,M) \to \O(M)$ has image of finite index;
\item suppose that $M^\perp$ is definite.  Then 
$\O(\L,M) \to \O(M)$ has finite kernel, and the centralizer $Z_{\O(\L)} H$ is a 
finite-index subgroup of $\O(\L,M)$.
\end{enumerate}
\end{proposition}

Our interest in this situation arises from the geometry of K3 surfaces.  Let
$X$ be a projective K3 surface defined over a field $F$, and let $K/F$ be a
Galois extension.  Let $\L = \Pic X_K$ with the intersection pairing, 
and let $H$ be the image of $\Gal(K/F)$ in $\O(\L)$.
If $X$ has a rational point over $F$, we have
$\Pic X_F = \L^H$.
(See Section~\ref{sec:ft} for more details. This statement holds slightly more generally: for example,
if $F$ is a number field and $X$ has points everywhere locally over $F$.)
The Hodge index theorem states that $\L$ has signature $(1,n)$ and $\L^H$
has signature $(1,m)$: therefore $M^\perp$ is definite.

Before giving the proof we first collect a few helpful statements, which are probably well known.
\begin{lemma}\label{sublemma}
Let $\L$ be a lattice and $M$ a sublattice.  Let $H$ be a subgroup of
$\O(\L)$ such that $\L^H$ is a lattice.  For groups $G' \subset G$, let
$Z_G(G')$ denote the centralizer of $G'$ in $G$.
\begin{enumerate}
\item\label{Mperp} $M^\perp$ is a sublattice of $\L$.
\item\label{Operp}  There is a natural injection
$d: \O(\L,M) \hookrightarrow \O(M) \oplus \O(M^\perp)$
with image of finite index.
\item\label{cent1} $Z_{\O(\L)} H$ is contained in $\O(\L,\L^H)$.
\item\label{cent2} The kernel of the map $\O(\L,\L^H) \to \O({\L^H}^\perp)$ is contained in
$Z_{\O(\L)} H$.
\end{enumerate}
\end{lemma}

\begin{proof}
To prove~(\ref{Mperp}), we just have to prove that the pairing on the subspace $M^\perp_\Q \subset \Lambda_\Q$ is non-degenerate; this is~\cite[Satz~1.2]{Eichler}.

Let $\phi \in \O(\L,M)$.  By definition $\phi$
restricts to an endomorphism of $M$.  Let $S_M$ be the saturation of
$M$ in $\L$.  Clearly $\phi(S_M) \subseteq S_M$.  Now,
$\phi^{-1}(S_M)$ has the same rank as $S_M$ and contains $S_M$, so it
is equal to $S_M$.  So if $\phi(S_M) \ne S_M$, then the image of $\phi$ does not
contain $S_M$, contradicting the hypothesis that $\phi$ is an automorphism of
$\L$.  Since $M$ is a
subgroup of finite index of $S_M$, this implies that $\#(S_M/M) =
\#(S_M/\phi(M))$.  But $\phi(M) \subseteq M$, so it follows that
$\phi(M) = M$.  Thus there is a map $\O(\L,M) \to \O(M)$.  Now, if
$\phi \in \O(\L,M)$ and $y \in M^\perp$, then $m \cdot \phi(y) =
\phi^{-1}(m) \cdot y = 0$ for all $m \in M$, so $\phi(y) \in M^\perp$, and
we get a map $\O(\L,M) \to \O(M^\perp)$ in the same way.  Combining these two
maps gives a map $d: \O(\L,M) \to \O(M) \oplus \O(M^\perp)$.

If $d(\phi) = 1$, then $d$ is the identity on $M \oplus M^\perp$, which is a
subgroup of $\L$ of finite index.  Because $\L$ is torsion-free, this
forces $\phi$ to be the identity.  To show that $\im d$ has finite image in
$(O(M) \oplus O(M^\perp))$,
et $n$ be the smallest positive integer such that
$n\L \subseteq M \oplus M^\perp$, and let $k = [M \oplus M^\perp:n\L]$.
Every element of $\O(M) \oplus \O(M^\perp)$ that fixes $n\L$ as a set
is in the image of
$d$, because the induced automorphism of $n\L$ extends to an automorphism of
$\L$ with the same action on $M \oplus M^\perp$.  Since there are only finitely
many subgroups of index $k$ in $M \oplus M^\perp$, the stabilizer of $n\L$ is of
finite index, and we have proved~(\ref{Operp}).

To prove (\ref{cent1}), let $\phi \in Z_{\O(\L)} H$,
and let $m \in \L^H$ and $h \in H$.  Then $h(m) = m$ and $\phi
\circ h = h \circ \phi$.  So $h(\phi(m)) = \phi(h(m)) = \phi(m)$,
establishing that $\phi(m) \in \L^H$.

Finally we prove (\ref{cent2}).  Choose $\phi$ in the kernel and $h \in H$, and let
$x \in \L$.  We will view $\phi$ and $h$ as automorphisms of $\L_\Q$.
In $\L_\Q$ we may write $x = x_1 + x_2$,
where $x_1 \in \L^H_\Q$ and $x_2 \in ({\L^H_\Q})^\perp$.
Then $h(\phi(x)) = h(\phi(x_1 + x_2)) = h(\phi(x_1)) + h(\phi(x_2)).$
However, $\phi(x_1) \in \L^H_\Q$, so $h(\phi(x_1)) = \phi(x_1)$,
and $\phi$ is in the kernel of the map to $\O({\L^H}^\perp)$, so $\phi(x_2)
= x_2$.  It follows that $h(\phi(x)) = \phi(x_1) + h(x_2)$.
Similarly, $\phi(h(x)) = \phi(h(x_1+x_2)) = \phi(h(x_1)) +
\phi(h(x_2)) = \phi(x_1) + h(x_2) = h(\phi(x))$, establishing that
$\phi$ commutes with $h$.
\end{proof}

\begin{proof}[Proof of Proposition~\ref{lattice-prop}]
The map $\O(\L,M) \to \O(M)$ is a
composition $\O(\L,M) \to \O(M) \oplus \O(M^\perp) \to \O(M)$.  In part (2) of
the lemma just proved we showed that the first map has image of finite index.
The second map is surjective, so the composition has image of finite index
as well.

We now suppose that $M^\perp$ is definite to prove the second statement.
Then $\O(M^\perp)$ is finite, so $\O(\L,M) \to \O(M)$ is a composition of
an injective map with a map with finite kernel and so its kernel is finite.
Let $K = \ker(\O(\L,M)
\to \O(M^\perp))$.  Then $Z_{\O(\L)} H/K$ has finite index in
$\O(\L,M)/K$, because both inject into the finite group
$\O(M^\perp)$.  Therefore $Z_{\O(\L)} H$ has finite index in $\O(\L,M)$
too.
\end{proof}

\section{Finiteness results for K3 surfaces}\label{sec:ft}

In this section we formulate and prove analogues of the statements of Theorem~\ref{thm:old} when $k$ is an arbitrary field.
We first look at the case of $k$ separably closed, which is straightforward.

\begin{lemma}\label{lem:sep}
Let $k = \ksep$ be a separably closed field, and let $\kbar$ be an algebraic closure of~$k$.  Let $X$ be a K3 surface over $k$, and let $\Xbar$ be the base change of $X$ to $\kbar$.  Then the natural maps $\Pic X \to \Pic \Xbar$ and $\Aut X \to \Aut \Xbar$ are isomorphisms.
\end{lemma}
\begin{proof}
As $X$ is projective, the Picard scheme $\Picsch_{X/k}$ exists, is separated and locally of finite type over $k$, and represents the sheaf $\Pic_{(X/k)(\et)}$, which is defined to be the sheafification on the big \'etale site over $k$ of the presheaf 
\[
T \mapsto \Pic(X \times_k T) / \Pic T
\]
(see~\cite[Theorem~9.4.8]{FAG}).  In particular, because $k$ and $\bar{k}$ are both separably closed, we have $\Picsch_{X/k}(k) = \Pic X$ and $\Picsch_{X/k}(\kbar) = \Pic \Xbar$.  
From $\H^1(X,\OO_X)=0$ it follows from~\cite[Theorem~9.5.11]{FAG} that $\Picsch_{X/k}$ is \'etale over $k$.
So every $\kbar$-point of $\Picsch_{X/k}$ is defined over $k$, and $\Pic X \to \Pic \Xbar$ is an isomorphism.

The functor taking a $k$-scheme $S$ to the group $\Aut(X \times_k S)$ is represented by a scheme $\Autsch_{X/k}$: see~\cite[Theorem~5.23]{FAG}. 
A standard argument in deformation theory shows that the tangent space at the identity element is isomorphic to $\H^0(X, T_X)$, where $T_X$ denotes the tangent sheaf on $X$.  
Indeed, an element of the tangent space is given by a morphism $S=\Spec k[\epsilon]/(\epsilon^2) \to \Autsch_{X/k}$ extending the morphism sending $\Spec k$ to the identity automorphism.  Such a morphism corresponds to an automorphism of $X \times_k S$ restricting to the identity on the central fibre.
By~\cite[Theorem~8.5.9]{FAG}, the set of these morphisms forms an affine space under $\H^0(X,T_X)$.
In our case, the group $\H^0(X, T_X)$ is zero \cite[Theorem~9.5.1]{huybrechts}, so the scheme $\Autsch_{X/k}$ is \'etale over $k$, and $\Aut X \to \Aut \Xbar$ is an isomorphism.
\end{proof}

\begin{corollary}\label{cor:sep-2}
In the situation of Lemma~\ref{lem:sep},  every $(-2)$-curve on $\Xbar$ is defined over $k$.
\end{corollary}
\begin{proof}
Let $\bar{C}$ be a $(-2)$-curve on $\Xbar$.  Then Lemma~\ref{lem:sep} shows that there is a line bundle $L$ on $X$ whose base change to $\Xbar$ is isomorphic to $\OO_{\Xbar}(\bar{C})$.  The Riemann--Roch theorem and flat base change give $h^0(X,L) = h^0(\Xbar,\OO_{\Xbar}(\bar{C})) = 1$.  So all nonzero sections of $L$ cut out the same divisor $C \subset X$ and the base change of $C$ to $\Xbar$ must coincide with $\bar{C}$.  In other words, $\bar{C}$ is defined over $k$.
\end{proof}

We now pass to the case of a general field.  Let $k$ be a field; fix
an algebraic closure $\kbar$ of $k$, and let $\ksep$ be the separable
closure of $k$ in $\kbar$.  Let $X$ be a K3 surface over $k$, and let
$\Xsep$ and $\Xbar$ denote the base changes of $X$ to $\ksep$ and
$\kbar$, respectively.  Write $\Gk = \Gal(\ksep/k)$.

The group $\Gk$ acts on $\Pic \Xsep$ preserving intersection numbers, giving a representation $\Gk \to \O(\Pic \Xsep)$.
Let $\Gk$ act on $\O(\Pic \Xsep)$ by conjugation, that is, such that $(\sigmaf{\sigma}{f})(x) = \sigma(f(\sigma^{-1} x))$ for all $x \in \Pic \Xsep$.
For a $(-2)$-class
$\alpha \in \Pic \Xsep$, denote the reflection in $\alpha$ by~$r_\alpha$; then we have $(\sigmaf{\sigma}{r_\alpha}) = r_{\sigma\alpha}$.
So the action of $\Gk$ on $\O(\Pic \Xsep)$ restricts to an action on $W(\Pic \Xsep)$.

\begin{defn}\label{defn:rx} Define $R_X$ to be the group
  $W(\Pic \Xsep)^\Gk$.
\end{defn}
  
Recall that $\Pic X$ is contained in, but not necessarily equal to, the fixed subgroup $(\Pic \Xsep)^\Gk$.
The Hochschild--Serre spectral sequence gives rise to an exact sequence
\[
0 \to \Pic X \to (\Pic \Xsep)^\Gk \to \Br k \to \Br X.
\]
If $X$ has a $k$-point, then evaluation at that point gives a left inverse to $\Br k \to \Br X$, showing that $\Pic X \to (\Pic \Xsep)^\Gk$ is an isomorphism.
In general this does not have to be true.  However, because $\Br k$ is torsion and $\Pic \Xsep$ is finitely generated, the above sequence shows that $\Pic X$ is of finite index in $(\Pic \Xsep)^\Gk$.

It is easy to see that the action of $R_X$ on $\Pic\Xsep$ preserves $(\Pic\Xsep)^\Gk$, but it is not immediately obvious that this action preserves $\Pic X$.  To show that this is the case, we use an explicit description of $R_X$ provided by a theorem of H\'ee and Lusztig, for which
Geck and Iancu gave a simple proof.  Before stating their theorem, we establish some notation and
conventions for Coxeter systems.

\begin{defn}\label{def:coxeter} Let $W$ be a group generated by a
  set $T \subset W$ of elements of order~$2$.  For $t_i, t_j \in T$, let
  $n_{i,j} = n_{j,i}$ be the order of $t_i t_j$ if $t_i t_j$ has finite order
  and~$0$ otherwise.  Suppose that the relations
  $t_i^2 = 1, (t_i t_j)^{n_{i,j}} = 1$ for $i,j$ with $n_{i,j} \ne 0$
  are a presentation of~$W$.  Then $(W,T)$ is a \emph{Coxeter system.}

  Let $G$ be a graph with vertices $T$ and such that $t_i, t_j$ are adjacent
  in $G$ if and only if $t_i$ does not commute with $t_j$; in this case, label
  the edge joining $t_i$ to $t_j$ with $n_{i,j} - 2$ for $n_{i,j} > 0$ and $0$
  otherwise.  We refer to $G$ as the
  \emph{Coxeter--Dynkin diagram} of $(W,T)$.  
  The Coxeter system $(W,T)$ is said to be \emph{irreducible} if its Coxeter--Dynkin diagram is connected.

  Let the {\em length} $\ell(w)$ of an element $w \in W$ be the length of a
  shortest word in the $t_i$ that represents it.  If $W$ is finite, there
  is $w_0 \in W$ such that $\ell(w_0) > \ell(w)$ for all
  $w \ne w_0 \in W$ (see \cite[Proposition 2.3.1]{bjorner-brenti}).
  We refer to $w_0$ as the \emph{longest element} of~$W$.

  Let $\sigma$ be a permutation of $T$.  Then there is at most one way to
  extend $\sigma$ to a homomorphism $W \to W$, because $T$ generates
  $W$.  If there is such an extension, it is an automorphism, because
  $\sigma^{-1}$ extends to its inverse, and we speak of it as the automorphism
  {\em induced by $\sigma$}.
\end{defn}

If $(W,T)$ is a Coxeter system, and $I$ is a subset of $T$, let $W_I$ denote the subgroup of $W$ generated by the elements of $I$.
Then $(W_I,I)$ is a Coxeter system: see \cite[Proposition 2.4.1 (i)]{bjorner-brenti}.

\begin{theorem}[\cite{geck-iancu}, Theorem~1]\label{thm:gi}
  Let $(W,T)$ be a Coxeter system.  Let
  $G$ be a group of permutations of $T$ that induce automorphisms of $W$.
  Let $F$ be the set of orbits $I \subset T$ for which
  $W_I$ is finite, and for $I \in F$ let $w_{I,0}$  be the longest element
  of $(W_I,I)$.  Then $(W^G,\{w_{I,0}: I \in F\})$ is a Coxeter system.
\end{theorem}

We will apply this theorem with $W=W(\Pic \Xsep)$ and $T$ being the set of reflections in $(-2)$-curves on $\Xsep$.

\begin{proposition}\label{rx}
Let $F$ be the set of Galois orbits $I$ of $(-2)$-curves on $\Xsep$ of the following two types:
\begin{enumerate}[(i)]
\item $I$ consists of disjoint $(-2)$-curves;
\item $I$ consists of disjoint pairs of $(-2)$-curves, each pair having intersection number 1.
\end{enumerate}
Then the following statements hold.
\begin{enumerate}
\item\label{rx1} For each $I \in F$, let $W_I$ be the subgroup of $W(\Pic \Xsep)$ generated by reflections in the classes of curves in $I$, and let $r_I$ be the longest element of the Coxeter system $(W_I,I)$.  Then $(R_X, \{r_I: I \in F\})$ is a Coxeter system.
\item\label{rx2} For each $I \in F$, let $C_I \in (\Pic \Xsep)^\Gk$ be the sum of the classes in $I$.  Then $r_I$ acts on $(\Pic \Xsep)^\Gk$ as reflection in the class $C_I$.
\item\label{rx3} The action of $R_X$ on $\Pic\Xsep$ preserves $\Pic X$.
\end{enumerate}
\end{proposition}

\begin{proof}
Let $I$ be a Galois orbit of $(-2)$-curves, and suppose that the subgroup $W_I$ is finite.
We will show that $I$ is of one of the two types described.
Firstly, no two $(-2)$-curves in $I$ have intersection number greater than $1$, for then the corresponding reflections would generate an infinite dihedral subgroup of $W_I$.
Since $W_I$ is finite, its Coxeter--Dynkin diagram is a finite union of trees~\cite[Exercise~1.4]{bjorner-brenti}.
In particular, it contains a vertex of degree $\le 1$.
However, the Galois group $\Gk$ acts transitively on the diagram, so we conclude that either every vertex has degree $0$, or every vertex has degree $1$.  These two possibilities correspond to the two types of orbits described.
Now~(\ref{rx1}) follows from Theorem~\ref{thm:gi}.

We prove~(\ref{rx2}) separately for the two types of orbits.
In the first case we have $I = \{ E_1, \dotsc, E_r \}$.  
The reflections in the $E_i$ all commute, so $W_I$ is isomorphic to the Coxeter group $A_1^r$.
The longest element is $r_I = r_{E_1} \circ \dotsb \circ r_{E_r}$.
For $D \in (\Pic \Xsep)^\Gk$, the intersection numbers $D \cdot E_i$ are all equal, and one calculates
\[
r_I(D) = D + (D \cdot E_1)(E_1 + \dotsb + E_r),
\]
that is, $r_I$ coincides on $(\Pic \Xsep)^\Gk$ with reflection in the class $C_I = E_1 + \dotsb E_r$, of self-intersection $-2r$.

In the second case, write $I = \{ E_1, E'_1, \dotsc, E_r, E'_r \}$, where $E_i \cdot E'_i = 1$ and the other intersection numbers are all zero.
The two reflections $r_{E_i}$ and $r_{E'_i}$ generate a subgroup isomorphic to the Coxeter group $A_2$, in which the longest element is
\[
r_I = r_{E_i} \circ r_{E'_i} \circ r_{E_i} = r_{E'_i} \circ r_{E_i} \circ r_{E'_i} = r_{E_i + E'_i}.
\]
Thus we have $W_I \cong A_2^r$ and the longest element in $W_I$ is the product of the longest elements in the factors $A_2$, that is, it equals $r_I = r_{E_1 + E'_1} \circ \dotsb \circ r_{E_r + E'_r}$.
For $D \in (\Pic \Xsep)^\Gk$, all the $D \cdot E_i$ and $D \cdot E'_i$ are equal, and we have
\[
r_I(D) = D + 2(D \cdot E_1)(E_1 + E'_1 + \dotsb + E_r + E'_r),
\]
which coincides with reflection of $D$ in the class $C_I = E_1 + E'_1 + \dotsb + E_r + E'_r$, of self-intersection $-2r$.

Finally, each class $C_I$ is, by construction, the class of a Galois-fixed divisor on $\Xsep$, so lies in $\Pic X$.  So in both cases the formula for $r_I$ given above shows that reflection in $C_I$ preserves $\Pic X$, and therefore the action of $R_X$ preserves $\Pic X$, proving~(\ref{rx3}).
\end{proof}

We now turn to the ample and nef cones.
As ampleness is a geometric property, and the nef cone is the closure of the ample cone over any base field, it follows that 
\[
\Amp(X) = \Amp(\Xbar) \cap (\Pic X)_\R \text{ and } \Nef(X) = \Nef(\Xbar) \cap (\Pic X)_\R,
\]
the intersections taking place inside $(\Pic\Xbar)_\R$.
The following result is well known when $k$ is algebraically closed (see~\cite[Corollary~8.2.11]{huybrechts}), and descends easily to arbitrary $k$.

\begin{proposition}\label{prop:fundamental-domain}
Let $X$ be a K3 surface over $k$.  The cone $\Nef(X) \cap \pos_X$ is a fundamental domain for the action of $R_X$ on the positive cone $\pos_X$, and this action is faithful.
\end{proposition}
\begin{proof}
We will prove two things: first, that every class in $\pos_X$ is $R_X$-equivalent to an element of $\Nef(X) \cap \pos_X$; and second, that the translates of $\Nef(X) \cap \pos_X$ by two distinct elements of $R_X$ meet only along their boundaries.
The second of these shows in particular that the action is faithful.
(When we refer to the boundary of $\Nef(X)$ or one of its translates, we mean the boundary within $(\Pic X)_\R$.  The boundary of $\Nef(X) \cap \pos_X$ in $\pos_X$ is just the boundary of $\Nef(X)$ intersected with $\pos_X$, so that distinction is not so important.)

To prove the first statement, let $D \in \pos_X$.  Suppose first that $D$ has trivial stabilizer in $W(\Pic \Xbar)$.  Then, by~\cite[Corollary~8.2.11]{huybrechts}, there exists a unique $g \in W(\Pic \Xbar) = W(\Pic \Xsep)$ such that $gD$ lies in the interior of $\Nef(\Xbar) \cap \pos_\Xbar$.  We claim that $g$ lies in $R_X$.  
For any $\sigma \in \Gk$, we have
\[
(\sigmaf{\sigma}{g}) D = \sigma(g(\sigma^{-1} D)) = \sigma (g(D)) \in \Nef(\Xbar) \cap \pos_\Xbar
\]
since the Galois action preserves the properties of being nef and positive.  By uniqueness of $g$, we conclude that $g = \sigmaf{\sigma}{g}$, that is, $g$ lies in $R_X$.  It then follows that $gD$ lies in $((\Pic \Xsep)_\R)^\Gk = (\Pic X)_\R$ and therefore in $\Nef(X) \cap \pos_X$.

Now suppose that $D \in \pos_X$ has non-trivial stabilizer.  Then $D$ lies on at least one of the walls defined by the action of $W(\Pic \Xbar)$ on $(\Pic\Xbar)_\R$; see~\cite[Section~8.2]{huybrechts}.
By~\cite[Proposition~8.2.4]{huybrechts}, the chamber structure of this group action is locally polyhedral within $\pos_\Xbar$, so a small enough neighbourhood of $D$ meets only finitely many chambers.
Also note that $\Pic X$ is not contained in any of the walls, because $X$ admits an ample divisor.  This allows us to construct a sequence $(D_i)_{i=1}^\infty$ of elements of $\pos_X$, tending to $D$ and all lying in the interior of the same chamber of $\pos_{\bar{X}}$.  As in the previous paragraph, there is a unique $g \in R_X$ satisfying $g D_i \in \Nef(X) \cap \pos_X$ for all $i$.  By continuity, $gD$ also lies in $\Nef(X) \cap \pos_X$.

We now prove the second statement, that $\Nef(X) \cap \pos_X$ intersects the translate by any non-trivial element of $R_X$ only in its boundary.
Suppose that $x \in \pos_X$ lies in the intersection $\Nef(X) \cap g \Nef(X)$, for some non-trivial $g \in R_X$.  By~\cite[Corollary~8.2.11]{huybrechts}, we see that $x$ lies in the boundary of $\Nef(\Xbar)$.
The following lemma shows that $x$ lies in the boundary of $\Nef(X)$.
\end{proof}

\begin{lemma}\label{lem:boundary}
Let $V$ be a real vector space, let $C \subset V$ be a closed convex cone, and let $S \subset V$ be a subspace having non-empty intersection with the interior of $C$.  Then we have
\[
\partial_S(C \cap S) = \partial_V(C) \cap S.
\]
\end{lemma}
\begin{proof}
The dual cone $C^* \subset V^*$ is defined by
\[
C^* = \{ \phi \in V^* \mid \phi(x) \ge 0 \text{ for all } x \in C \}.
\]
By the supporting hyperplane theorem, $C^*$ has the property that a point $x \in C$ lies in $\partial_V(C)$ if and only if there exists a non-zero $\phi \in C^*$ satisfying $\phi(x)=0$.
(The hyperplane $\phi=0$ is called a \emph{supporting hyperplane} of $C$ at $x$.)
Let $f \colon V^* \to S^*$ be the natural restriction map; then the dual $(C \cap S)^*$ is equal to $f(C^*)$ (see~\cite[Corollary~16.3.2]{rockafellar}).

Let $x$ be a point of $\partial_V(C) \cap S$.  Then there is a supporting hyperplane to $C$ at $x$, that is, there exists a non-zero $\phi \in C^*$ satisfying $\phi(x)=0$.  The condition that $S$ meet the interior of $C$ implies that $\phi$ does not vanish identically on $S$, so $f(\phi)$ is non-zero.  Thus $f(\phi)$ is a non-zero element of $(C \cap S)^*$ vanishing at $x$, so $x$ lies in $\partial_S(C \cap S)$.

Conversely, suppose that $x$ lies in $\partial_S(C \cap S)$.  Then there is a supporting hyperplane to $C \cap S$ at $x$, that is, there exists a non-zero $\psi \in (C \cap S)^*$ satisfying $\psi(x)=0$.  Let $\phi \in C^*$ satisfy $f(\phi)=\psi$; then we have $\phi(x)=0$ and so $x \in \partial_V(C)$.
\end{proof}

\begin{remark}\label{walls}
We can also give an explicit description of the walls of $\Nef(X) \cap \pos_X$.
According to~\cite[Corollary~8.1.6]{huybrechts}, a class in $\pos_X$ lies in $\Nef(X)$ if and only if it has non-negative intersection number with every $(-2)$-curve on $\Xsep$, or, equivalently, with every Galois orbit of $(-2)$-curves.  The question is to determine which Galois orbits are superfluous, and which actually define walls of $\Nef(X) \cap \pos_X$.

Let $I$ be a Galois orbit of $(-2)$-curves on $\Xsep$, and suppose that the subgroup $W_I \subset W(\Pic \Xsep)$ generated by reflections in the elements of $I$ is finite, that is, $I$ is as described in Proposition~\ref{rx}.
The longest element of $W_I$ acts on $\Pic X$ by reflection in the class $C_I = \sum_{E \in I} E$, which has negative self-intersection.
The hyperplane orthogonal to $C_I$ is a wall of $\Nef(X) \cap \pos_X$: by the same argument as in~\cite{sterk}, given a class $D\in \Amp(X)$, the class 
\[
D - \frac{D \cdot C_I}{C_I \cdot C_I} C_I
\]
is orthogonal to $C_I$ but has positive intersection number with all $(-2)$-curves outside $I$.

On the other hand, let $I$ be a Galois orbit of $(-2)$-curves such that the subgroup $W_I$ is infinite.
Then, for a $(-2)$-curve $C \in I$, the hyperplane orthogonal to $C$ in $\Pic(\Xsep)_\R$ does not meet $\Nef(X) \cap \pos_X$, as the following argument shows.  
Let $x$ be a class in $\pos_X$ orthogonal to $C$; then $x$ is also orthogonal to all the other curves in $I$, and so is fixed by $W_I$.
Therefore $x$ is also orthogonal to the infinitely many images of $C$ under the action of $W_I$, 
contradicting the fact that the chamber structure induced by $W(\Pic\Xsep)$ on $\pos_\Xsep$ is locally polyhedral.
\end{remark}

Next we would like to prove an analogue of Theorem \ref{thm:old} (2).
However, while the authors do not have an explicit counterexample, there seems to be no reason for the image of $R_X$ in $\O(\Pic X)$ to be normal.
Instead, we will see that there is a natural homomorphism from a semidirect product $\Aut X \ltimes R_X$ to $\O(\Pic X)$ having finite kernel and image of finite index.

Note that the natural action of $\Aut X \subset \Aut\Xbar$ on $\O(\Pic\Xbar)$ by conjugation fixes $W(\Pic\Xbar)$ and commutes with the Galois action, so fixes $R_X$.
This gives an action of $\Aut X$ on $R_X$, and a homomorphism from the associated semidirect product $\Aut X \ltimes R_X$ to $\O(\Pic\Xbar)$.
Since $\Aut X$ and $R_X$ both fix $\Pic X$, we also obtain a homomorphism $\Aut X \ltimes R_X \to \O(\Pic X)$.
\begin{proposition}\label{main-result}
Let $X$ be a K3 surface over a field $k$ of characteristic different from $2$.  Then the natural map
\[
\Aut X \ltimes R_X \to \O(\Pic X)
\]
has finite kernel and image of finite index.
\end{proposition}

\begin{remark}
In the literature, this statement appears in various different but equivalent forms.
\begin{itemize}
\item Because the action of $\Aut X$ fixes the ample cone, the image of $\Aut X$ in $\O(\Pic X)$ meets $R_X$ only in the identity element.
Therefore the finiteness of the kernel in Proposition~\ref{main-result} is equivalent to the finiteness of the kernel of $\Aut X \to \O(\Pic X)$.
\item Let $\Aut_s X \subset \Aut X$ be the subgroup of symplectic automorphisms~\cite[Definition~15.1.1]{huybrechts}.
Over an algebraically closed field $k$ of characteristic zero, the induced map from $\Aut_s X$ to $\O(\Pic X)$ is injective, so that $\Aut_s X \ltimes W(\Pic X)$ can be viewed as a subgroup of $\O(\Pic X)$.
Instead of our Proposition~\ref{main-result}, Huybrechts~\cite[Theorem~15.2.6]{huybrechts} makes the statement that $\Aut_s X \ltimes W(\Pic X)$ has finite index in $\O(\Pic X)$.
This is equivalent to our formulation, because $\Aut_s X$ is of finite index in $\Aut X$.
However, when $k$ is not algebraically closed there is no reason for $\Aut_s X \to \O(\Pic X)$ to be injective, so there is no advantage to stating the result in terms of $\Aut_s X$.
\item Lieblich and Maulik~\cite[Theorem~6.1]{lm} define $\Gamma \subset \O(\Pic X)$ to be the subgroup of elements preserving the ample cone, and then show that $\Aut X \to \Gamma$ has finite kernel and cokernel.  
Over an algebraically closed base field $k$, the subgroups $\Gamma$ and $W(\Pic X)$ generate $\O(\Pic X)$: any element of $\O(\Pic X)$ permutes the $(-2)$-classes in $\Pic X$ and therefore takes the ample cone to one of its translates under $W(\Pic X)$; so composing with an element of $W(\Pic X)$ gives an element of $\Gamma$.
Therefore that condition is also equivalent to the condition of Proposition~\ref{main-result}.
\end{itemize}
\end{remark}

Before proving Proposition~\ref{main-result}, we first state two lemmas which are well known in the case of abelian groups. 
\begin{lemma}\label{lem:kercoker}
Let $G$ be a finite group, and $f \colon A \to B$ a homomorphism of (possibly non-commutative) $G$-modules having finite kernel and image of finite index.  Then the induced homomorphism $f^G \colon A^G \to B^G$ also has finite kernel and image of finite index.
\end{lemma}
\begin{proof}
The kernel of $f^G$ is contained in the kernel of $f$, so is finite.  For the statement about the image, consider the short exact sequence
\[
0 \to \ker f \to A \xrightarrow{f} \im f \to 0.
\]
The associated exact sequence of cohomology gives
\[
0 \to (\ker f)^G \to A^G \xrightarrow{f^G} (\im f)^G \to \H^1(G, \ker f)
\]
and $\H^1(G,\ker f)$ is a finite set, showing that $\im(f^G)$ is of finite index in $(\im f)^G$.  On the other hand, we claim that $(\im f)^G$ is of finite index in $B^G$.  Indeed, by hypothesis $\im(f)$ is of finite index in $B$, and this property is preserved on intersecting with the subgroup $B^G$.  Thus $\im(f^G)$ is of finite index in $B^G$.
\end{proof}

The following lemma is standard and easy to prove; we state it for reference.

\begin{lemma}\label{lem:comp}
Let $f \colon A \to B$ and $g \colon B \to C$ be two homomorphisms of groups.
\begin{enumerate}
\item If $f$ and $g$ both have finite kernel and image of finite index, then so does the composition $g \circ f$.
\item If $g \circ f$ has finite kernel and image of finite index, and $g$ has finite kernel, then $f$ has finite kernel and image of finite index.
\item If $g \circ f$ has finite kernel and image of finite index, and $f$ has image of finite index, then $g$ has finite kernel and image of finite index.
\end{enumerate}
\end{lemma}
\begin{proof}
We leave this as an exercise for the reader.
\end{proof}

\begin{proof}[Proof of Proposition~\ref{main-result}]
By~\cite[Theorem~15.2.6]{huybrechts} in characteristic zero, or~\cite{lm} in characteristic $p > 2$, the natural map 
\[
\Aut \Xbar \to \O(\Pic \Xbar) / W(\Pic \Xbar)
\]
has finite kernel and image of finite index.  Lemma~\ref{lem:sep} shows that the same is true for $\Aut \Xsep \to \O(\Pic \Xsep) / W(\Pic \Xsep)$.
The action of $\Gk$ on all of these groups factors through a finite quotient, so Lemma~\ref{lem:kercoker} shows that the induced homomorphism
\begin{equation}\label{eq:hom}
(\Aut \Xsep)^\Gk \to (\O(\Pic \Xsep) / W(\Pic \Xsep))^\Gk
\end{equation}
also has finite kernel and image of finite index.  Because the automorphism group functor is representable, we have $(\Aut \Xsep)^\Gk = \Aut X$.  

There is an exact sequence
\[
1 \to R_X \to \O(\Pic\Xsep)^\Gk \to (\O(\Pic \Xsep) / W(\Pic \Xsep))^\Gk
\]
and the homomorphism~\eqref{eq:hom} factors through $\O(\Pic\Xsep)^\Gk$, and hence through the injective map 
\[
\O(\Pic\Xsep)^\Gk/R_X \to (\O(\Pic \Xsep) / W(\Pic \Xsep))^\Gk.
\]
Therefore, by Lemma~\ref{lem:comp}, the map 
\[
\Aut X \to \O(\Pic \Xsep)^\Gk / R_X
\]
has finite kernel and image of finite index.  Since the image of $\Aut X$ in $\O(\Pic X^s)^\Gk$ meets $R_X$ only in the identity element, this shows that the natural map
\begin{equation}\label{eq:hom2}
\Aut X \ltimes R_X \to \O(\Pic X^s)^\Gk
\end{equation}
also has finite kernel and image of finite index.

We apply Proposition~\ref{lattice-prop} with $\Lambda = \Pic \Xsep$ and $H$ being the image of $\Gk$ in $\O(\Pic\Xsep))$.
The centralizer $Z_{\O(\Lambda)}(H)$ is then $\O(\Pic \Xsep)^\Gk$, so part (2) of Proposition~\ref{lattice-prop} shows that $\O(\Pic \Xsep)^\Gk$ is of finite index in $\O(\Pic \Xsep,(\Pic \Xsep)^\Gk)$. Parts~(1) and (2) of Proposition~\ref{lattice-prop} combined show that the map 
\[
\O(\Pic \Xsep,(\Pic \Xsep)^\Gk) \to \O((\Pic \Xsep)^\Gk)
\]
has finite kernel and image of finite index; by Lemma~\ref{lem:comp} so does the map $\O(\Pic\Xsep)^\Gk \to \O((\Pic\Xsep)^\Gk)$.

Composing with~\eqref{eq:hom2} and applying Lemma~\ref{lem:comp} shows that
\[
\Aut X \ltimes R_X \to \O((\Pic\Xsep)^\Gk)
\]
has finite kernel and image of finite index.
Since the actions of both $\Aut X$ and $R_X$ on $(\Pic \Xsep)^\Gk$ preserve $\Pic X$, this last map factors as 
\begin{equation}\label{eq:hom4}
\Aut(X) \ltimes R_X \to  \O((\Pic \Xsep)^\Gk,\Pic X) \to  \O((\Pic \Xsep)^\Gk).
\end{equation}
As the second map in this composition is clearly injective, Lemma~\ref{lem:comp} shows that the first one has finite kernel and image of finite index.

Finally, $\Pic X$ is of finite index in $(\Pic \Xsep)^\Gk$, so its orthogonal complement in the latter is trivial.  Lemma~\ref{sublemma}(\ref{Operp}) shows that the natural map 
\[
\O((\Pic \Xsep)^\Gk, \Pic X) \to \O(\Pic X)
\] 
is injective, and its image has finite index. Composing with the first map of \eqref{eq:hom4} and applying Lemma~\ref{lem:comp} again, we deduce that $\Aut X \ltimes R_X \to \O(\Pic X)$ has finite kernel and image of finite index.
\end{proof}

\begin{remark}
The finiteness of the kernel can also be proved directly, by the same proof as over an algebraically closed field.
\end{remark}

Having proved Proposition~\ref{main-result}, we can deduce the remaining results exactly as in the classical case.  Define the cone $\Nef^e(X)$ to be the real convex hull of $\Nef(X) \cap \Pic X$.

\begin{corollary}\label{rat-poly}
The action of $\Aut X$ on $\Nef^e(X)$ admits a rational polyhedral fundamental domain.
\end{corollary}
\begin{proof}
This is as in the case of $k=\C$; we briefly recall the argument of Sterk~\cite{sterk}, making the necessary adjustments.

Let $\Lambda$ be a lattice of signature $(1,\rho-1)$ and let $\Gamma \subset \O(\Lambda_\R)$ be an arithmetic subgroup (for example, a subgroup of finite index in $\O(\Lambda)$).
Let $C \subset \Lambda_\R$ be one of the two components of $\{ x \in \Lambda_\R \mid (x \cdot x) < 0 \}$; this is a self-adjoint homogeneous cone~\cite[Remark~1.11]{ash}.  Let $C_+$ be the convex hull of $\overline{C} \cap \Lambda_\Q$.
Pick any $y \in C \cap \Lambda$; the argument of~\cite[p.~511]{sterk} shows that the set
\[
\Pi = \{ x \in C_+ \mid ( \gamma x \cdot y ) \ge ( x \cdot y ) \text{ for all } \gamma \in \Gamma \}
\]
is rational polyhedral (for this, the reference to Proposition~11 of the first edition of~\cite{ash} has become Proposition~5.22 in the second edition) and is a fundamental domain for the action of $\Gamma$ on $C_+$.
(Sterk does not explicitly prove that two translates of $\Pi$ intersect only in their boundaries, but this is easy to show from the description above.)
In our case, applying this with $\Lambda = \Pic X$ and $\Gamma$ being the image of $\Aut X \ltimes R_X \to \O(\Pic X)$ gives a rational polyhedral fundamental domain $\Pi$ for the action of $\Aut X \ltimes R_X$ on $(\pos_X)_+$.

If we choose $y$ to be an ample class in $\Pic X$, then the resulting $\Pi$ is contained in $\Nef(X)$, as we now show.
Let $I$ be a Galois orbit of $(-2)$-curves on $\Xsep$ such that the corresponding group $W_I \subset W(\Pic\Xsep)$ is finite.
Proposition~\ref{rx} states that the longest element $w$ of $W_I$ acts on $\Pic X$ as reflection in the class $C_I = \sum_{E \in I} E$,
and that these elements generate $R_X$.
Taking $\gamma=w$ in the definition of $\Pi$ shows that $\Pi$ is contained in the half-space $\{x \mid x.C_I \ge 0 \}$.
As this holds for all such $I$, Remark~\ref{walls} shows that $\Pi$ is contained in $\Nef(X)$.

We conclude as in~\cite{sterk}.
If $x$ is a class in $\Nef^e(X)$ then, since $\Pi$ is a fundamental domain for the action of $\Aut X \ltimes R_X$ on $(\pos_X)_+$,
we can find $\phi \in \Aut X$ and $r \in R_X$ such that $r \phi(x)$ lies in $\Pi$.  But now $\phi(x)$ and $r \phi(x)$ both lie in $\Nef^e(X)$, so they are equal and lie in $\Pi$.
This shows that $\Pi$ is a fundamental domain for the action of $\Aut X$ on $\Nef^e(X)$.
\end{proof}

\begin{corollary}\label{finite-class-orbits}
\begin{enumerate}
\item\label{g0} There are only finitely many $\Aut X$-orbits of $k$-rational $(-2)$-curves on $X$.
\item\label{g1} There are only finitely many $\Aut X$-orbits of primitive Picard classes of irreducible curves on $X$ of arithmetic genus $1$.
\item\label{g2} For $g \ge 2$, there are only finitely many $\Aut X$-orbits of Picard classes of irreducible curves on $X$ of arithmetic genus $g$.
\end{enumerate}
\end{corollary}

\begin{proof}
Let $\Pi$ be a rational polyhedral fundamental domain for the action of $\Aut X$ on $\Nef^e(X)$, as in Corollary~\ref{rat-poly}.
Every $k$-rational $(-2)$-curve on $X$ defines a wall of $\Nef(X)$, by Remark~\ref{walls}.
Since $\Pi$ meets only finitely many walls of $\Nef(X)$, this proves~(\ref{g0}).

Gordan's Lemma states that the integral points of the dual cone of a rational polyhedral convex cone form a finitely generated monoid.
Applying this to the dual cone of $\Pi$, let $D_1, \dotsc, D_r$ be a minimal set of generators for $\Pi \cap \Pic X$.
Since these all lie in $\Nef(X)$, we have $D_i \cdot D_i \ge 0$ for all $i$, and $D_i \cdot D_j  > 0$ for $i \neq j$.
As observed in~\cite{sterk}, this implies that, for any $n > 0$, there are only finitely many classes in $\Pi \cap \Pic X$ of self-intersection $n$;
and there are only finitely many primitive classes in $\Pi \cap \Pic X$ of self-intersection zero.
The class of an irreducible curve of arithmetic genus $g \ge 1$ has self-intersection $2g-2$ and therefore lies in $\Nef(X)$, so this proves~(\ref{g1}) and~(\ref{g2}).
\end{proof}

\begin{remark}
It is not true that every irreducible curve on $X$ of arithmetic genus $0$ is a $k$-rational $(-2)$-curve.
However, there are not many possibilities, as we now show.
Let $C$ be such a curve, and let $C_1, \dotsc, C_r$ be the geometric components of $C$; the Galois group $\Gk$ acts transitively on them.
In order to achieve $C^2=-2$, we must have $C_i^2<0$ for all $i$, so each $C_i$ is a $(-2)$-curve.
Consider the intersection matrix $(C_i \cdot C_j)$: the sum of the entries in the matrix is $C^2=-2$, and the Galois action shows that every row sum is the same.  Therefore the number $r$ of rows divides $2$, and there are only two options: $r=1$, so that $C$ is a rational $(-2)$-curve; or $r=2$, and $C=C_1 \cup C_2$ is the union of two conjugate $(-2)$-curves meeting transversely in a single point.
By Corollary~\ref{rat-poly}, both types of curves define walls of $\Nef(X)$, so in fact both types fall into finitely many orbits under the action of $\Aut(X)$.

Both types occur on the surfaces considered in Section~\ref{finite-diagonal}
below.  Indeed, the line $x = y, z = w$ on the surface $x^4 - y^4 = c(z^4 - w^4)$
is a rational $(-2)$-curve, while the line $x = y, z = iw$ meets its conjugate
transversely in a single point.
\end{remark}

\begin{remark}
In the case $g=1$, the condition that the class be primitive cannot be omitted, for the following reason.
Take for example $k=\Q$, and suppose that~$X$ admits an elliptic fibration $\pi \colon X \to \P^1$.  A general fibre of such a fibration is a smooth, geometrically irreducible curve $E$ of genus $1$, whose class in $\Pic X$ has self-intersection $0$ and is primitive.  (In fact, the class is even primitive in $\Pic\Xbar$.)
If $s \in \P^1$ is a point of degree $m>1$, then in general the fibre $\pi^{-1}(s)$ will be an irreducible curve on $X$ of arithmetic genus $1$, linearly equivalent to $mE$.
As $m$ varies, this construction gives infinitely many such classes that are clearly not $\Aut X$-equivalent.
\end{remark}

\section{Examples}\label{sec:eg}
In this section we give three examples 
that illustrate some of the theory developed up to this point.
First, we will give a K3 surface
$X$ with finite automorphism group, even though all K3 surfaces $V$ over
$\bar \Q$ with $\Pic V \cong \Pic X$ have infinite automorphism group.
Second, we will construct a surface $Y$ with finite automorphism group, even though, 
for all extensions $k/\Q$, all K3 surfaces $V$ over $\bar \Q$ with
$\Pic V \cong \Pic Y_k$ have infinite automorphism group.  For a third example,
we will prove that the quartic surface $Z$ in $\P^3$
defined by $x^4 - y^4 = c(z^4 - w^4)$ over a field $k$ of characteristic $0$
has finite automorphism group when $c \in k$
is such that the Galois group of the field of definition of the Picard group
has degree $16$, the largest possible.

\subsection{First example}\label{finite-from-overlattice}

We construct a surface $X$ such that $\Aut \bar{X}$ is finite, and thus so is $\Aut X$,
while any K3 surface over $\bar{\Q}$ having the same Picard lattice as $X$ has infinite automorphism group.
In contrast, over an algebraically closed field of characteristic not equal to
$2$, finiteness of the automorphism group depends only on the isomorphism type of the Picard lattice.  This is an immediate consequence of statement $2$ of Theorem~\ref{thm:old}.

Let $M$ and $N$ be the block diagonal matrices
\[
M = \begin{pmatrix} 0&1&0\\1&0&0\\0&0&-8 \end{pmatrix}, \qquad
N = \left( \begin{array}{cc|c} 0&1\\1&0\\ \hline &&-2I_4 \end{array}\right),
\]
where $I_4$ is the $4\times 4$ identity matrix.
Let $L_N$ be a lattice with basis $( e_1, \dotsc, e_6 )$ and Gram matrix $N$ with respect to that basis.
Let $L_M \subset L_N$ be the sublattice generated by $e_1$, $e_2$ and $e_3+e_4+e_5+e_6$.  The Gram matrix for $L_M$ with respect to this basis is $M$.
The surface $X$ that we will construct will have compatible isomorphisms $\Pic\bar{X} \cong L_N$ and $\Pic X \cong L_M$.

\begin{proposition}\label{prop:aut-x-inf}
  Let $V$ be a K3 surface over $\bar \Q$ with
  $\Pic V \cong L_M$.
  Then $\Aut V$ is infinite.
\end{proposition}

\begin{proof}
  There is an obvious embedding of the hyperbolic lattice $U$ with Gram matrix
  $\left( \begin{smallmatrix}0&1\\1&0\end{smallmatrix} \right)$
  into $L_M \cong \Pic V$.
  By~\cite[Remark~11.1.4]{huybrechts}, this implies that there is an elliptic fibration $\pi \colon V \to \P^1$ with a section.
  Let $E$ be the class of a fibre of $\pi$ and $O$ the class of a section; then $E$ and $O$ generate a sublattice of $\Pic V$ isomorphic to $U$ (though not necessarily the obvious one).
  If $\pi$ were to have a reducible fibre, then any component of that fibre other than the one meeting $O$ would lie in ${\langle E,O \rangle}^\perp$ and have self-intersection $-2$.
  However, ${\langle E,O \rangle}$ is a lattice of rank $2$ and determinant $-1$,
  so the determinant of ${\langle E,O \rangle}^\perp$ is the negative of that
  of $\Pic V$ and
  its rank is $2$ less.  Thus it is generated by a single vector of norm $-8$
  and so there are no reducible fibres.
  It follows by the Shioda--Tate formula
  \cite[Corollary 11.3.4]{huybrechts}
  that the Mordell--Weil group of the fibration has rank $1$.
  Translation by a non-torsion section gives an automorphism of $V$ of infinite order.
\end{proof}

\begin{remark}
In fact, Shimada~\cite[Remark~9.3]{shimada} has proved that this translation, and negation in the Mordell--Weil group, generate the whole of $\Aut V$.
\end{remark}

In contrast with Proposition~\ref{prop:aut-x-inf}, we will see that there exist K3 surfaces over $\Q$ with Picard lattice isomorphic to $L_M$ for which the automorphism group is finite.  We use the standard Kodaira symbols for reducible
fibres of elliptic fibrations, as in \cite[IV.9, Table 4.1]{silverman}.

\begin{proposition}\label{prop:autxfinite}
  Let $V$ be a K3 surface over 
  an algebraically closed field 
  with an elliptic fibration $\pi \colon V \to \P^1$ that has a section.  Suppose that $\pi$ has four fibres each of type either $\textit{I}_2$ or $\textit{III}$.  Then the following statements are equivalent.
\begin{enumerate}
\item\label{pic6} The rank of $\Pic V$ is at most $6$.
\item\label{picLN} The Picard lattice $\Pic V$ is isomorphic to $L_N$.
\item\label{mw} The Mordell--Weil group associated to $\pi$ is trivial, and there are only four reducible fibres.
\end{enumerate}
If these equivalent conditions hold, then $\Aut V$ is finite.
\end{proposition}

Before proving the proposition, we state and prove a lemma.

\begin{lemma}\label{lem:cant-divide-four}
  Let $V$ be a K3 surface containing four pairwise disjoint smooth rational curves
  $C_1, C_2, C_3, C_4$.  Then the class $\sum_{i=1}^4 [C_i]$ cannot be divided by
  $2$ in $\Pic V$.
\end{lemma}

\begin{proof}
  Suppose otherwise, and let $D = (\sum_{i=1}^4 [C_i])/2 \in \Pic V$.  Then $D^2 = -2$, so either
  $D$ or $-D$ is effective; since $2D$ is effective, we deduce that $D$ is effective.
  Now $(D,[C_i]) = -1$ for
  $1 \le i \le 4$, so the $C_i$ are base components of
  the linear system $|D|$ and $D - \sum_{i=1}^4 [C_i] = -D$ is also effective, giving a contradiction.
\end{proof}

\begin{proof}[Proof of Proposition~\ref{prop:autxfinite}]
The equivalence of~(\ref{picLN}) and~(\ref{mw}) is proved as follows.  
By~\cite[Proposition~11.3.2]{huybrechts}, there is an exact sequence
\[
0 \to A \to \Pic V \to G \to 0,
\]
where $A \subset \Pic V$ is the subgroup generated by vertical divisors and a chosen section, and $G$ is the Mordell--Weil group of the elliptic fibration.
Let $E$ be a fibre of $\pi$, let $O$ be a chosen section, and let $W_1, \dotsc, W_4$ be the components of the four given fibres that do not meet $O$.
The classes $[E], [E]+[O], [W_1], \dotsc, [W_4]$ lie in $A$ and have intersection matrix equal to $N$, so they generate a sublattice of $A$ isomorphic to $L_N$.
If $G$ is trivial and there are no other reducible fibres, then these six classes generate $\Pic V$ and so we have $\Pic V \cong L_N$.
Conversely, if $\Pic V$ is isomorphic to $L_N$ then these six classes must generate $\Pic V$, so $G$ is trivial.  Also, there can be no further reducible fibres, for the class of a curve in such a fibre would be independent of the given generators of $\Pic V$.

The implication~(\ref{picLN})$\Rightarrow$(\ref{pic6}) is trivial; we now prove~(\ref{pic6})$\Rightarrow$(\ref{picLN}).
  As above, we have an embedding $L_N \hookrightarrow \Pic V$, so $\Pic V$ must have rank exactly $6$.  
  Since $\Pic V$ has rank $6$ this embedding has finite index, and
  we must prove it to be an isomorphism.  The determinant of $N$ is~$-16$;
  the square of the index $[\Pic V:L_N]$ must divide this, so the index is
  $1, 2,$ or $4$.  If it is not $1$, there is some element of $L_N$ that
  can be divided by $2$ in $\Pic V$ but not in $L_N$.  We take this element
  to be of the form $a[E] + b[O] + \sum_{i=1}^4 c_i [W_i]$, where all the
  coefficients are $0$ or $1$ and not all are $0$.
  Then $a = b = 0$, for otherwise the intersection number with
  $[O]$ or $[E]$ would be odd; and all the $c_i$ must be equal, because the
  self-intersection of any divisor on a K3 surface is even and hence that
  of any divisor that can be divided by $2$ is a multiple of $8$.  Thus
  all $c_i$ are equal to~$1$.  However, Lemma~\ref{lem:cant-divide-four}
  shows that $\sum_{i=1}^4 [W_i]$ is not divisible by $2$.
  This proves $\Pic V \cong L_N$.
  
Finally, Nikulin~\cite[Theorem~3.1]{nikulin} has listed the finitely many possibilities for $\Pic V$ of rank $\ge 6$ that give rise to finite automorphism groups.  The lattice $L_N \cong U \oplus 4A_1$ is in the list, showing that $\Aut V$ is finite.
\end{proof}

\begin{proposition}\label{prop:auttfinite}
Let $X$ be a K3 surface over a field $k$ with an elliptic fibration $\pi \colon X \to \P^1$ that has a section.
Suppose that $\pi$ has four conjugate fibres of type $\textit{I}_2$ or $\textit{III}$
and that the rank of $\Pic \bar{X}$ is at most $6$.
Then there are compatible isomorphisms $\Pic\bar{X} \cong L_N$
and $\Pic X \cong L_M$, and $\Aut X$ is finite.
\end{proposition}
\begin{proof}
Applying Proposition~\ref{prop:autxfinite} to the surface $\bar{X}$ shows that $\Pic \bar{X}$ is isomorphic to $L_N$.
More precisely, the proof shows that there is an isomorphism $\Pic\bar{X} \cong L_N$ that identifies the basis $(e_1, \dotsc, e_6)$ of $L_N$ with the basis $([E], [E]+[O], [W_1], \dotsc, [W_4])$ of $\Pic \bar{X}$, where $E$ is a fibre of $\pi$ and $O$ is a section, and $W_1, \dotsc, W_4$ are the components of the four reducible fibres that do not meet $O$.
Since by Proposition~\ref{prop:autxfinite} there are no other reducible fibres,
the Galois action permutes $W_1, \dotsc, W_4$ transitively and so the Galois-invariant subgroup is identified with $L_M$.
As $X$ contains a $k$-rational curve $O$ of genus $0$, it has rational points over $k$ and hence $\Pic X = (\Pic \bar{X})^{\Gamma_\Q} = L_M$.

Proposition~\ref{prop:autxfinite} also shows that $\Aut\bar{X}$ is finite, and \emph{a fortiori} that $\Aut X$ is finite.
\end{proof}

We now construct a K3 surface $X$ over $\Q$ satisfying the conditions of Proposition~\ref{prop:auttfinite},
as the Jacobian of a genus-$1$ fibration on a quartic $U$ in $\P^3$.
See \cite[Section 11.4]{huybrechts} for the properties of the Jacobian of
a genus-$1$ fibration on a K3 surface.  

Let $U$ be a smooth quartic surface in $\P^3$ over $\Q$ containing a line $L$.
Projection away from $L$ induces a morphism $\pi_L \colon U \to \P^1$ whose fibres are the residual intersections with $U$ of planes containing $L$.
The generic fibre is a smooth curve of genus $1$, and the induced morphism $L \to \P^1$ has degree $3$.

\begin{proposition}\label{prop:four-a1}
  Let $U$ be a smooth quartic surface in $\P^3_\Q$ containing a
  rational line $L$ and a Galois orbit $\mathcal{L}$ of four lines that meet $L$.  
  Suppose in addition that each of the four planes containing $L$ and a line in $\mathcal{L}$ meets $U$ in one further component, which is a smooth conic.
  Then $\pi_L$
  has four conjugate fibres of type $\textit{I}_2$ or $\textit{III}$.
  
  Let $X \to \P^1$ be the relative Jacobian of $\pi_L$, and suppose in addition that $\Pic(\bar{U})$ has rank at most $6$.
  Then there are compatible isomorphisms $\Pic \bar{X} \cong L_N$ and $\Pic X \cong L_M$, and $\Aut X$ is finite. 
\end{proposition}

\begin{proof}
  Let $H$ be one of the four conjugate planes containing $L$ and a line $L' \in \mathcal{L}$,
  and let $C$ denote the residual smooth conic.
  The union $L' \cup C$ is a fibre of $\pi_L$.
  We have $L' \cdot C = 2$, because $L', C$ are a line and a conic in the
  plane $H$.  Either $L'$ is tangent to $C$, in which case we have a fibre of type
  $\textit{III}$, or they intersect in two distinct points, and the fibre is of type
  $\textit{I}_2$.  The same description holds for the other three planes that are
  Galois conjugates of $H$.
  
  The relative Jacobian $X$ is a K3 surface~\cite[Proposition~11.4.5]{huybrechts} 
  that has the same geometric Picard number as $U$~\cite[Corollary~11.4.7 and the discussion following it]{huybrechts},
  and $X \to \P^1$ has the same geometric fibres as $\pi_L$~\cite[Chapter~11, equation~(4.1)]{huybrechts}.
  Now apply Proposition~\ref{prop:auttfinite}.
\end{proof}

A very general surface $U$ constructed according to this
proposition has Picard group generated by the classes of $L$, the lines in $\mathcal{L}$ and a fibre of~$\pi_L$.
This does not imply that such a surface exists over $\Q$,
but we will see that it is not difficult to find an example.

\begin{example}\label{ex:construct-t}
We claim that the surface $U \subset \P^3_\Q$ given by the equation
\begin{multline}\label{eq:R}
-2x^3z - 3x^2yz - 3y^3z + x^2z^2 - 3xyz^2 + 2y^2z^2 + xz^3 + yz^3 - 13x^3w \\ + 24x^2yw 
 - 13xy^2w + 8y^3w - x^2zw + 51xz^2w - 37x^2w^2 + 47xyw^2 - 16y^2w^2 \\ + 111xzw^2 
 - 38yzw^2 - 57z^2w^2 - 227xw^3 + 24yw^3 - 94zw^3 + 303w^4 = 0
\end{multline}
satisfies the conditions of Proposition~\ref{prop:four-a1}.
Indeed, $U$ contains the line $L$ defined by $w=z=0$ and, for any $\alpha$ satisfying $\alpha^4+\alpha-1 = 0$,
$U$ contains the line $L_\alpha$ through the point 
$(-\alpha^3 + \alpha^2 - \alpha + 1 : 1 : 0 : 0)$ on $L$ and the point
$(-\alpha^3 - \alpha + 1 : -\alpha^3 - \alpha^2 + 1 : \alpha^3 - \alpha^2 + \alpha + 1 : 1)$.
The plane $H_\alpha$ containing $L$ and $L_\alpha$ is given by $z=(\alpha^3-\alpha^2+\alpha+1)w$, and one
verifies that $U \cap H_\alpha$ consists of $L$, $L_\alpha$ and a smooth conic.

Let $U_3 \subset \P^3_{\F_3}$ be the surface defined by the equation~\eqref{eq:R}, which is smooth.
Let~$\bar{U}_3$ be the base change of $U_3$ to an algebraic closure of $\F_3$, and $F \colon \bar{U}_3 \to \bar{U}_3$
the geometric Frobenius morphism, defined by $(x,y,z,w) \mapsto (x^3,y^3,z^3,w^3)$.
Choose a prime $\ell \neq 3$ and let $F^*$
be the endomorphism of $\H^2(\bar{U}_3, \Q_\ell(1))$ induced by $F$.
By~\cite[Proposition~6.2]{rvl}
the Picard rank of $\bar{U}$ is bounded above by that of $\bar{U}_3$,
which in turn is at most the number of eigenvalues of $F^*$ that are roots of unity
by~\cite[Corollary~6.4]{rvl}.
As in~\cite{rvl}, we find the characteristic polynomial of $F^*$ by counting points on $U_3$.
The results are shown in Table~\ref{counts}.
\begin{table}[tb]
\caption{Point counts on the surface $U_3$}\label{counts}
\begin{tabular}{c|cccccccc}
$n$ & $1$ & $2$ & $3$ & $4$ & $5$ & $6$ & $7$ & $8$ \\ \hline
$\#U_3(\F_{3^n})$ & $16$ & $94$ & $676$ & $7066$ & $60076$ & $533818$ & $4785076$ & $43101802$
\end{tabular}
\end{table}

From the Lefschetz fixed point formula we find that the trace of the $n$th power of Frobenius acting on $\H^2_\et(\bar{U}_3,\Q_\ell)$ is equal to $\#U_3(\F_{3^n}) - 3^{2n} - 1$. The trace on the Tate twist $\H^2_\et(\bar{U}_3,\Q_\ell(1))$ is obtained by dividing by $3^n$, while on the subspace~$V \subset \H^2_\et(\bar{U}_3,\Q_\ell(1))$ generated by $H, L$, and the four lines $L_\alpha$, the trace $t_n$ is equal to $6$ if $n$ is a multiple of $4$, and equal to $2$ if $n$ is not a multiple of $4$; hence, on the $16$-dimensional quotient 
$Q=\H^2_\et(\bar{U}_3,\Q_\ell(1))/V$, the trace equals $\#U_3(\F_{3^n})/3^n - 3^{n} - 3^{-n} - t_n$.
These traces are sums of powers of eigenvalues, and we use the Newton identities to compute the elementary symmetric polynomials in these eigenvalues, which are the coefficients of the characteristic polynomial $f$ of Frobenius acting on~$Q$. This yields the first half of the coefficients of $f$, including the middle coefficient, which turns out to be nonzero. This implies that the sign in the functional equation $t^{16}f(1/t) = \pm f(t)$ is $+1$, so this functional equation determines $f$, which we calculate to be
\[
f = \tfrac{1}{3}\big(3t^{16} + t^{14} + 4t^{13} + 2t^{10} - 2t^8 + 2t^6 + 4t^3 + t^2 + 3\big).
\]
As the characteristic polynomial of Frobenius acting on $V$ is $(t-1)^2(t^4-1)$, we find that the characteristic polynomial of Frobenius acting on $\H^2_\et(\bar{U}_3,\Q_\ell(1))$ is $(t-1)^2(t^4-1)f$. 
The polynomial $3f \in \Z[t]$ is irreducible, primitive and not monic, so its roots are not roots of unity. Thus we obtain an upper bound of~$6$ for the Picard rank of $\bar{U}$.
\end{example}

\begin{remark}
We found $U$ by first fixing the lines $L$ and $L'$.  The space of rational quartic polynomials vanishing on $L$ and on
the conjugates of $L'$ is easily checked to have dimension $14$.
  (This is the expected dimension.  The space of quartic polynomials has
  dimension $35$ and vanishing on a line is equivalent to vanishing on
  $5$ points of the line and so imposes $5$ conditions.  However, if two lines
  meet in a point, that point gives the same condition for both lines.
  With four pairs of intersecting lines, we therefore
  expect the dimension to be $35 - 5 \times 5 + 4 = 14$.)
We randomly chose elements of this space until we obtained one defining a smooth
surface with good reduction at $3$ and suitable characteristic polynomial of Frobenius.
\end{remark}

\begin{lemma}\label{nine-curves} The surface $\bar X$ has exactly $9$ smooth
  rational curves.
\end{lemma}

\begin{proof}
  Let the $W_i$ be the components of the reducible fibres that meet $O$.
  In our basis, the known
  rational curves have classes $[O], [W_i], [E]-[W_i]$ for $1 \le i \le 4$.
  Suppose that there is another rational curve, of class
  $C = a[E] + b[O] + \sum_{i=1}^4 c_i [W_i]$.
  It must have non-negative intersection with the known curves, which implies
  the inequalities
  $$a \ge b \ge 0, \quad c_i \le 0, \quad b + c_i \ge 0 \qquad (1 \le i \le 4).$$
  Let $m = \min(\{c_i\})$: then
  $ab \ge b^2 \ge 4m^2$.  We thus find
  $$C^2 = 2ab - 2\sum_{i=1}^4 c_i^2 \ge 2b^2 - 8m^2 \ge 0,$$
  which contradicts the fact that $C^2 = -2$.
\end{proof}

\begin{remark}\label{more-curves-on-r}
  The surface $\bar U$, on the other hand, has infinitely many smooth
  rational curves.
  Indeed, by \cite[Chapter~11, equation~(4.5)]{huybrechts}, the Picard lattice
  of $\bar U$ has discriminant $-144$.  However, the list of Picard lattices
  of rank $\ge 6$ giving finite automorphism group in
  \cite[Theorem~3.1]{nikulin} does not include any lattices of rank $6$ and
  discriminant $-144$.  Hence $\Aut \bar U$ is infinite.  Let
  $\mathcal C$ be the union of $\{L\}$ and the set of components of the
  reducible fibres of $\pi_L$.  By construction $\mathcal C$ 
  spans a subgroup of $\Pic \bar U$ of finite index (in fact, by computing the
  discriminant one checks that $\mathcal C$
  generates $\Pic \bar U$).  It follows that the stabilizer of $\mathcal C$ in
  $\Aut \bar U$ is finite and hence that the orbit is infinite.
\end{remark}

\begin{remark}\label{prop:min-deg}
We now explain why we do not construct the Jacobian of $U$ directly as a
smooth surface in a projective space.

  If $H$ is an ample divisor class on a K3 surface $X$ with
  $\Pic X \cong L_N$, then $H^2 \ge 16$.
  To see this, let $H = a[E] + b[O] + \sum_{i=1}^4 c_i [W_i]$ be such a class.
Since $H$ is
  ample, all of $H \cdot [O], H \cdot [W_i], H \cdot ([E] - [W_i])$ must be
  positive: that is,
  $$a - 2b + \sum_{i=1}^4 c_i \ge 1, \quad b - 2c_i \ge 1, \quad c_i \ge 1.$$
  So $b \ge 3$ and $b+2c_i \ge 5$.
  We thus find
  $$\begin{aligned}
    H^2 &= 2ab - 2b^2 + 2 \sum_{i=1}^4 bc_i - 2 \sum_{i=1}^4 c_i^2 \cr
    &\ge 2b(2b - \sum_{i=1}^4 c_i + 1) - 2b^2 + 2 \sum_{i=1}^4 bc_i - 2 \sum_{i=1}^4 c_i^2 \cr
    &= 2b^2 + 2b - 2 \sum_{i=1}^4 c_i^2 \cr
    &= 2b + \frac{\sum_{i=1}^4 (b-2c_i)(b+2c_i)}2\cr
    &\ge 6 + 4\cdot 1 \cdot 5/2 = 16.\cr
  \end{aligned}
  $$

  Conversely, such a surface has
  an ample divisor class of self-intersection $16$.  To see this, note that
  equality is attained in the above
  with $a = 3, b = 3, c_i = 1$, and this gives an ample
  divisor class by \cite[Corollary 8.1.7]{huybrechts}.
\end{remark}

Thus an ample divisor class cannot give an embedding into $\P^n$ for
$n < 9$.
  
\subsection{Second example}\label{finite-by-reflections}
Now we give an example that is perhaps more surprising: a K3 surface
$Y/\Q$ for which $\Aut Y$ is finite even though, for all field
extensions $L/\Q$, a K3 surface over $\bar \Q$ whose Picard lattice is
isomorphic to $\Pic Y_L$ would have infinite automorphism group.

\begin{defn}
  Let
  \[
  M = \begin{pmatrix}10&0\cr 0&-4\cr \end{pmatrix}, \quad
  N = \begin{pmatrix}10&0&0\cr 0&-2&0\cr 0&0&-2\cr\end{pmatrix},
  \]
    and let $L_M, L_N$ be the lattices with Gram matrices $M, N$ respectively.
\end{defn}

We will choose $Y$ to be a K3 surface whose Picard lattice over $\Q$
is isomorphic to $L_M$, while over
$\bar \Q$, and indeed over a certain quadratic extension $K_Y$, the
Picard lattice is isomorphic to $L_N$.  The Galois
group will act through the quotient $\Z/2\Z$ by exchanging the second
and third generators.

\begin{proposition}\label{prop:describe-u}
  Let $Y$ be a K3 surface in $\P^4_\Q$ given as the intersection of a quadric
  and a cubic, containing two disjoint Galois-conjugate conics
  $C_1, C_2$ defined over a quadratic field $K_Y$,
  and having Picard number $3$ over $\bar \Q$.  Then
  $\Pic Y \cong L_M$ and $\Pic Y_{\bar \Q} \cong L_N$.
\end{proposition}

\begin{proof}
  Let $H$ be a hyperplane section.  Then the divisors
  $H + C_1 + C_2, C_1, C_2$
  have the intersection matrix $N$, whose determinant is $-40$.
  Thus we have an embedding
  $L_N \hookrightarrow \Pic Y_{\bar \Q}$ whose image has index $1$ or $2$.  If it
  is $2$, then either $[H] + [C_1]$ or $[H] + [C_2]$ can be divided by $2$.
  If $[H] + [C_1] \sim 2D$, then, letting $\sigma$ be an extension of the nontrivial
  automorphism of the field of definition of $C_1$ to that of $D$, we have
  $[H] + [C_2] \sim 2D^\sigma$, and so both classes can be divided by $2$ and the
  index is a multiple of $4$, a contradiction.  Similarly, $[H] + [C_2]$ cannot
  be divided by $2$.  We conclude that the index is $1$: that is,
  $\Pic Y_{\bar \Q} \cong L_N$.  There are Galois-invariant divisors of the
  classes $[H] + [C_1] + [C_2], [C_1] + [C_2]$, and these span the invariant
  subspace, so they generate $\Pic Y$.  Their intersection matrix is $M$.
\end{proof}

\begin{remark}\label{rem:bdry-ample} The divisor class $[H]$ is very ample,
  because it is the hyperplane class on the smooth projective surface $Y$.
  The divisor class $D = [H]+[C_1]+[C_2]$ is not ample, but we will show that it is nef.
  Indeed, for any irreducible curve $C$ other than $C_1$ and $C_2$, the intersection numbers $H \cdot C$, $C_1 \cdot C$, $C_2 \cdot C$ are all non-negative, while $D \cdot [C_1]$ and $D \cdot [C_2]$ are both zero. 
  In fact, one can show that  
  $D$ is the hyperplane class for a
  model of $Y$ as a surface of degree $10$ in $\P^6$ with two ordinary
  double points ($A_1$ singularities).
\end{remark}

\begin{proposition}\label{prop:aut-y-z-infinite}
\begin{enumerate}
\item\label{itemLM}  The group $\O(L_M)$ is infinite. If  $\alpha \in \O(L_M)$ has infinite order and $A$ is a normal subgroup of
  $\O(L_M)$ containing $\alpha$ then $\O(L_M)/A$ is finite.
\item\label{item:Y}
Let $V$ be a K3 surface over an algebraically closed field having Picard lattice isomorphic to $L_M$.  Then $\Aut V$ is infinite.  
\item\label{item:Z}
Let $W$ be a K3 surface over an algebraically closed field having Picard lattice isomorphic to $L_N$.  Then $\Aut W$ is infinite.
\end{enumerate}
\end{proposition}

\begin{proof} We first prove~(\ref{itemLM}).
  To find $\O(L_M)$, we consider the quadratic form
  $10x^2 - 4y^2 = -N(2y + \sqrt{10} x)$ associated to $M$.  Its automorphism
  group is generated
  by the sign changes $(x,y) \to (x,-y)$ and $(x,y) \to (-x,y)$ and by
  multiplication by a generator of the group of totally positive units
  of $\O_{\Q(\sqrt{10})}$.
  (This generator is $19 + 6 \sqrt{10}$ and it takes $(x,y)$ to
  $(19x+12y,30x+19y)$.)
  Thus $\O(L_M)$ has a subgroup of finite
  index isomorphic to $\Z$, so the quotient by any infinite normal subgroup is finite.
  
  We continue with (\ref{item:Y}).
    Working mod $5$ we see that $\Pic V$ has no vectors of
  norm $-2$, whence $V$ has no rational curves.  Thus 
  $\Aut V$ coincides up to finite index with the infinite group $\O(\Pic V) = \O(L_M)$:
  this follows from Theorem \ref{thm:old} (2), because $W(\Pic V)$ is trivial.

  We now turn to~(\ref{item:Z}).  Let $D_1, D_2, D_3$ be divisors on $W$ whose
  intersection matrix is $N$, and let $D = D_1-2D_2-D_3$.  Since $D^2 = 0$
  and $D$ is primitive,
  for some $\alpha \in \O(L_N)$ there is a genus-$1$ fibration $\pi$ with
  fibres of class $\alpha(D)$.  
  There is no section of $\pi$
  (indeed, no two curves on $W$ have odd intersection), but
  there is a $2$-section.  The Jacobian $J$ of $\pi$ is a K3 surface of
  Picard number $3$ \cite[Corollary~11.4.7]{huybrechts} on
  which the determinant of the intersection pairing is $(\det N)/2^2 = -10$
  \cite[Equation~(11.4.5)]{huybrechts}.
  As in Proposition \ref{prop:aut-x-inf}, this shows that $J$ has no reducible fibres: the non-identity component of a reducible fibre would be orthogonal to the classes of both a fibre and the zero-section, and have self-intersection $-2$, which is incompatible with the required determinant.  The Shioda--Tate formula then shows that the Mordell--Weil group of $J$ has rank $1$.
  This Mordell--Weil group acts faithfully on $W$, so it follows as in
  Proposition \ref{prop:aut-x-inf} that $\Aut W$ is infinite.
\end{proof}

Proposition~\ref{prop:aut-y-z-infinite} states that any K3 surface over an algebraically closed field having Picard lattice isomorphic to either $L_M$ or $L_N$ has infinite automorphism group.
In contrast, the following proposition shows that a K3 surface over $\Q$ having Picard lattice isomorphic to $L_M$ or $L_N$ over any algebraic extension of $\Q$ may have finite automorphism group.
Indeed, the last part of this section will be devoted to finding an example of such a surface.

\begin{proposition}\label{prop:auts-finite}
  Let $Y$ be a K3 surface as in Proposition \ref{prop:describe-u}.
  Then $\Aut Y$ is finite.
\end{proposition}

\begin{proof}
  Let $D_1 = [H] + [C_1] + [C_2]$ and $D_2 = [C_1] + [C_2]$, so that
  $D_1, D_2$ have the intersection matrix $M$.  Since $C_1, C_2$ are disjoint
  conjugate rational curves,
  the product $r_1$ of the reflections in their classes is in our reflection group
  $R_Y$ (Definition~\ref{defn:rx}).
  In the given basis, $r_1$ acts on $\Pic Y$ with matrix
  \[A_1 = \begin{pmatrix}1&0\cr 0&-1\cr\end{pmatrix}.\]
  
  We now show that there is a smooth rational curve in the class
  $F = 6D_1 - 9[C_1] - 10[C_2]$, defined over the same field $K_Y$ as the
  $C_i$.  The class $F$ is effective, since it has self-intersection $-2$
  and positive intersection with the very ample divisor class $[H]$.  An
  irreducible curve of self-intersection $-2$ must be rational and smooth, so it suffices
  to show that $F$ has no nontrivial expression as a sum of effective
  classes.

  Let $E$ be an effective divisor with $[E] = aD_1 - b_1 [C_1] - b_2 [C_2]$.
  By Remark~\ref{rem:bdry-ample}, $D_1$ is nef, so
  $0 \le [E] \cdot D_1 = 10a$,
  and therefore $a \ge 0$.  
  Moreover, if $E$ is irreducible, then the inequality $E^2 \geq -2$ yields
  $10a^2\geq 2b_1^2+2b_2^2-2$, and if furthermore we have $a>0$, then the 
  inequality $C_i \cdot E \geq 0$ gives $b_i \geq 0$. 

  We write the class $F$ as a sum $F = \sum_{i=1}^s [E_i]$ of classes of irreducible 
  effective divisors $E_1, \ldots, E_s$, and we write $[E_i] = a_iD_i - b_{i,1}[C_1] - b_{i,2}[C_2]$. 
  We saw above that $a_i \geq 0$ for each $i$, so from $\sum_i a_i=6$, we find 
  $a_i \leq 6$ for each $i$. Furthermore, from $\sum_i (b_{i,1}+b_{i,2}-3a_i) = 9+10 - 3\cdot 6 = 1 >0$, 
  we conclude that there is an $i$ with $b_{i,1}+b_{i,2} > 3a_i$, and hence $b_{i,1}+b_{i,2} \geq 3a_i+1$. 
  Again from the above, we find for such $i$ that
  \[
  10a_i^2 \geq 2b_{i,1}^2+2b_{i,2}^2-2 \geq (b_{i,1}+b_{i,2})^2 - 2 \geq (3a_i+1)^2-2, 
  \]
  which implies $a_i \geq 6$, so $a_i = 6$. The left- and right-hand sides of the sequence of inequalities above 
  then differ by only $1 = 360-359$, so we also find $b_{i,1}+b_{i,2} = 3a_i+1=19$  (not $-19$ because $b_{i,j} \ge 0$) and 
  $(2b_{i,1}^2+2b_{i,2}^2) - (b_{i,1}+b_{i,2})^2 \leq 1$.
  Equivalently, $(b_{i,1}-b_{i,2})^2 \leq 1$, which yields 
  $[E_i] = F$ or $[E_i] = F' = 6D_1 - 10[C_1]-9[C_2]$.  But $F - [E_i]$ is
  effective by definition, and $F - F' = [C_1] - [C_2]$ is not effective (the intersection with the ample divisor $H$ is $0$), so
  $[E_i] = F$ and
  $F$ is indeed the class of 
  an irreducible curve, say $R$. 
  
  Let $\sigma$ be an automorphism of $\bar \Q$ that exchanges $C_1$ and $C_2$.
  Then the class $F^\sigma=[R^\sigma]$ is $6D_1 - 10[C_1] - 9[C_2]$, and one checks
  immediately that $F \cdot F^\sigma = 0$.
  The product $r_2$ of the reflections in $F$ and $F^\sigma$ has matrix
  \[A_2 = \begin{pmatrix}721&456 \cr -1140&-721\cr \end{pmatrix}\]
  with respect to the 
  basis $(D_1,D_2)$ of $\Pic Y$.  Now, $r_1 r_2$ is an element of $R_Y$ of infinite order.
  From Proposition \ref{prop:aut-y-z-infinite} (\ref{itemLM}), it follows that the
  quotient $\O(\Pic Y)/R_Y$ is finite and hence that
  $\Aut Y$ is finite.
\end{proof}

\begin{remark} Given equations for $Y$, it is quite practical to make
  the curve $R$ used in the proof explicit.
  For any variety $V \subset \P^4$, let $I(V)$ be the ideal in the homogeneous coordinate ring of
  $\P^4$ of polynomials that vanish on $V$.  Then the dimension of the
  degree-$6$ part of $((I(Y)+I(C_1)^3) \cap (I(Y)+I(C_2)^4))/I(Y)$ is $1$;
  let $f_6$ be a generator, so that the divisor cut out on $Y$ by $f_6$ is
  of the form $C + 3C_1 + 4C_2$.  Using Magma one checks that $C$ is indeed an
  irreducible curve of arithmetic genus $0$.
\end{remark}

\begin{remark}
  The numerical properties of $F$ follow from the fact that $(19,6)$ is a
  solution to the Pell equation $x^2 - 10y^2 = 1$.  More generally,
  let $k$ be even and suppose that $D_1, [C_1], [C_2]$ are classes of
  pairwise disjoint divisors on a surface with $D_1^2 = k$ and $C_1^2 = C_2^2$,
  where $C_1$ and $C_2$ constitute a Galois orbit.
  Suppose further that
  $m^2 - kn^2 = 1$: clearly $m$ is odd.  Let $D$ be the divisor class
  $nD_1 - \lfloor \frac{m}{2} \rfloor [C_1] - \lceil \frac{m}{2} \rceil [C_2]$
  and let $D^\sigma$ be its Galois conjugate.  Then
  $D^2 = kn^2 - 2 \left(\frac{m-1}{2}\right)^2 - 2 \left(\frac{m+1}{2}\right)^2 = -2$
  and $D \cdot D^\sigma = kn^2 - 4\frac{(m-1)(m+1)}{4} = 0$.
\end{remark}

  The remainder of this section will be devoted to constructing a K3 surface~$Y$ satisfying the
  conditions of Proposition~\ref{prop:describe-u}.  The verification will be
  more complicated than that of Example~\ref{ex:construct-t}, because the
  Picard number over $\bar \Q$ in this example is odd, so that the reduction
  map cannot induce an isomorphism of $\Pic Y_{\bar \Q}$ to
  $\Pic Y_{\bar \F_p}$.  We use the following
  proposition.

  \begin{proposition}\label{prop:how-to-find-u}
    Let $Y$ be a smooth intersection of a quadric and a cubic in $\P^4(\Q)$
    containing two disjoint conics
    $C_1, C_2$ that are defined and conjugate over a quadratic field $K_Y$.
    Let $p \ne 2, q$ be primes such that:
    \begin{enumerate}
    \item $Y$ has good reduction at $p$ and $q$;
    \item the reduction $Y_{\F_p}$ contains a line $L$ disjoint from
      the reductions of $C_1, C_2$, and its Picard group over $\bar\F_p$
      has rank $4$;
    \item the reduction $Y_{\F_q}$ base changed to $\bar \F_q$
      contains no lines disjoint from the reductions of $C_1$ and $C_2$.
    \end{enumerate}
    Then $Y_{\bar \Q}$ has Picard number $3$ and hence $Y$ satisfies the
    conditions of Proposition \ref{prop:describe-u}.
  \end{proposition}

  \begin{remark}
    To speak of $Y_{\F_p}$ we first need to choose a model of $Y$ over $\Z$.
    When we give equations with integral coefficients for $Y$ over $\Q$, it is
    understood that the model over $\Z$ is defined by the same equations.
    In each case it may be checked 
    that the ideal defining this model is generated by the intersection of
    the ideal defining $Y$ over $\Q$ with the ring of polynomials in the
    same variables with integral coefficients.
  \end{remark}

  \begin{proof}
    Let $s_p$ be the specialization homomorphism
    $\Pic Y_{\bar \Q} \to \Pic Y_{\bar \F_p}$, which is injective by
    \cite[Proposition 6.2]{rvl}.  Then by
    \cite[Theorem 1.4]{elsenhans-jahnel} the cokernel of $s_p$ is torsion-free.
    Note that specialization preserves intersection numbers \cite[Section 20.3]{fulton}.
    If $s_p$ is surjective, the class of $L$
    is in the image.  The class $L_0$ with $s_p(L_0) = [L]$ has positive
    intersection with the hyperplane class and self-intersection $-2$, so
    it is effective; its intersection with the hyperplane class is $1$, 
    so it is the class of a line on $Y_{\bar \Q}$.  
    We have $[L] \cdot s_p(C_i) = 0$, so $L_0 \cdot C_i = 0$ and
    $s_q(L_0) \cdot s_q(C_i) = 0$ as well.  This contradicts hypothesis (3).
    It follows that
    $\rk \Pic Y_{\bar \Q} < \rk \Pic Y_{\bar \F_p} = 4$.  On the other
    hand, the hyperplane class in $\P^4$ and the classes of $C_1, C_2$ are
    independent, so $\rk \Pic Y_{\bar \Q} \ge 3$.
  \end{proof}

  In our construction, we will use $p = 3, q = 5$. 
  We will take the conics $C_1, C_2$
  to be defined over $\Q(\sqrt{19})$, the smallest quadratic field in which
  both $p$ and $q$ split.  Let $\rho = \sqrt{19}$, and denote the coordinates
  on $\P^4$ by $x_0, \dots, x_4$.  We choose the conic $C_1$ defined by
  \[
  \begin{aligned}
    &(-2\rho + 2)x_0 + x_2 + (\rho - 2)x_3 + (-\rho - 2)x_4 = \\
    &(2\rho + 1)x_0 + (\rho + 2)x_1 + 2x_2 + (\rho + 2)x_3 + (-\rho + 2)x_4 = \\
    & \quad (-\rho - 2)x_0^2 + (-2\rho + 1)x_0x_1 + (-2\rho - 1)x_1^2 + (-\rho + 2)x_0x_2 + (-\rho + 2)x_1x_2 \\
    & \qquad + (\rho - 2)x_2^2 + (2\rho + 1)x_0x_3 + (2\rho - 2)x_1x_3 + (2\rho + 2)x_2x_3 - 2x_3^2 \\
    & \qquad + (\rho + 2)x_0x_4 + (-\rho + 2)x_1x_4 + (-\rho + 1)x_2x_4 - x_3x_4 + 2x_4^2 = 0
  \end{aligned}
  \]
  and let $C_2$ be its Galois conjugate.
  Next, let $L$ be the line in $\P^4(\F_3)$ through
  $(1:1:2:0:1),(2:1:1:1:0)$
and let $C_{3,1}, C_{3,2}$ be the reductions of $C_1$ and $C_2$, respectively, at one of the two primes
above~$3$.

\begin{proposition}\label{prop:char-3}
The surface $Y_3$ in $\P^4(\F_3)$ defined by $F_3=G_3=0$, where
\[
\begin{aligned}
  F_3 &= 2x_0x_1 + x_1x_2 + x_2^2 + 2x_0x_3 + x_1x_3 + x_3^2 + 2x_0x_4 + x_3x_4 
  + 2x_4^2, \\
  G_3 &= 2x_0x_1^2 + 2x_0^2x_2 + 2x_0x_1x_2 + 2x_1^2x_2 + 2x_0x_2^2 +  2x_1x_2^2 + 2x_2^3 + 2x_0x_1x_3 + x_1^2x_3 \\
  & \quad + 2x_0x_2x_3 +  2x_1x_2x_3 + 2x_0x_3^2 + 2x_2x_3^2 + 2x_3^3 + x_0^2x_4 + 
  2x_1^2x_4 + 2x_0x_2x_4 \\
  & \quad + x_2^2x_4 + 2x_0x_3x_4 + 2x_2x_3x_4 + 
  x_3^2x_4 + 2x_0x_4^2 + x_1x_4^2 + x_2x_4^2 + x_4^3,
\end{aligned}
\]
is smooth, contains $L, C_{3,1}$, and $C_{3,2}$, and has Picard number $4$.
\end{proposition}

\begin{proof} The verification that $Y_3$ is smooth and contains the given curves is routine.
To see that $Y_3$ has Picard number $4$, we proceed as in Example~\ref{ex:construct-t}.
In order to count the number of points on $Y_3$ over $\F_{3^n}$ for $1\leq n \leq 9$, we 
write $Y_3$ as an elliptic surface using the fibration of Picard class $H - C_{3,1}$, for which $L$ is a
section, and we sum the number of points of the fibres. We obtain point counts as in Table~\ref{counts2}.

\begin{table}[tb]
\caption{Point counts on the surface $Y_3$}\label{counts2}
\begin{tabular}{c|c}
$n$ & $\#Y_3(\F_{3^n})$ \\\hline
$1$ & $18$ \\
$2$ & $104$ \\
$3$ & $846$ \\
$4$ & $6776$ \\
$5$ & $59658$ \\
$6$ & $532694$ \\
$7$ & $4790811$ \\
$8$ & $43068056$ \\
$9$ & $387398079$
\end{tabular}
\end{table}

As in Example~\ref{ex:construct-t}, we find that the trace of the $n$th power of Frobenius acting on $\H^2_\et(\bar{Y}_3,\Q_\ell(1))$ is equal to $\#Y_3(\F_{3^n})/3^n - 3^{n} - 3^{-n}$, while on the subspace~$V \subset \H^2_\et(\bar{Y}_3,\Q_\ell(1))$ generated by $H, C_{3,1}, C_{3,2}$, and $L$, the trace is equal to $4$; hence, on the $18$-dimensional quotient 
$Q=\H^2_\et(\bar{Y}_3,\Q_\ell(1))/V$, the trace equals $\#Y_3(\F_{3^n})/3^n - 3^{n} - 3^{-n} - 4$. 
Again, these traces, together with the Newton identities, determine the coefficients of the characteristic polynomial $f$ of Frobenius acting on~$Q$. The point counts in Table~\ref{counts2} yield the first half of the coefficients of $f$, including the middle coefficient, which turns out to be nonzero. This implies that the sign in the functional equation $t^{18}f(1/t) = \pm f(t)$ is $+1$, 
so this functional equation determines $f$, which we calculate to be
\begin{align*}
f = \tfrac{1}{3} \big( 3t^{18} &- 4t^{17} + 5t^{16} - 4t^{15} + 4t^{14} - 4t^{13} + 5t^{12} - 5t^{11} + 5t^{10}  \\
          &- 6t^9 + 5t^8 - 5t^7 + 5t^6 - 4t^5 + 4t^4 - 4t^3 + 5t^2 - 4t + 3 \big). 
\end{align*}
As the characteristic polynomial of Frobenius acting on $V$ is $(t-1)^4$, we find that the characteristic polynomial of Frobenius acting on 
$\H^2_\et(\bar{Y}_3,\Q_\ell(1))$ is $(t-1)^4f$. 
The polynomial $3f$ is irreducible, primitive and not monic, so its roots are not roots of unity. 
Thus we obtain an upper bound of $4$ for the Picard rank of $\bar{Y}_3$.
The rank is equal to $4$ as the Picard group contains the linearly independent classes of $H$, 
$L$, $C_{3,1}$, and $C_{3,2}$.
\end{proof}

Now let $C_{5,1}, C_{5,2}$ be the reductions of $C_1$ and $C_2$, respectively, at one of the two primes above $5$.

\begin{proposition}\label{prop:char-5}
Let $Y_5$ be the surface in $\P^4_{\F_5}$ defined by $F_5=G_5=0$, where
\[
\begin{aligned}
  F_5 &= 2x_0^2 + x_0x_1 + 3x_1^2 + 2x_0x_2 + 2x_2^2 + 2x_1x_3 + 2x_2x_3 + 
  3x_3^2 + 3x_0x_4 + 2x_2x_4, \\
  G_5 &= x_0^2x_1 + 3x_0x_1^2 + 2x_1^3 + 4x_0x_1x_2 + x_1^2x_2 + 2x_1x_2^2 + 
  3x_2^3 + 2x_0^2x_3 + 4x_1^2x_3 \\
  & \quad + 2x_0x_3^2 + 2x_1x_3^2 + 
  2x_2x_3^2 + 2x_3^3 + x_0x_1x_4 + 4x_1^2x_4 + x_0x_2x_4 + 
  2x_2^2x_4 \\
  & \quad + 3x_1x_3x_4 + 2x_2x_3x_4 + 3x_3^2x_4 + 3x_0x_4^2 + 
  2x_1x_4^2 + 3x_2x_4^2 + x_3x_4^2.
\end{aligned}
\]
Then $Y_5$ is smooth.  It contains $C_{5,1}$ and $C_{5,2}$, but no lines over $\bar \F_5$
that are disjoint from $C_{5,1}$.
\end{proposition}

\begin{proof}
  Again, it is easy to check that $Y_5$ is smooth and that the $C_{5,i}$ are on $Y_5$.
It is also easy to verify directly that
$Y_5$ contains no lines defined over $\F_5$ by checking that no line through
two distinct $\F_5$-points of $Y_5$ is contained in $Y_5$;
however, this is not sufficient.
Suppose that $L$ is a line contained in $Y_5$ and disjoint from $C_{5,1}$.
We embed $Y_5$ in $\P^5_{\F_5}$ by the divisor class $H + [C_{5,1}]$, where
$H$ is the hyperplane class in $\P^4_{\F_5}$.  
In this embedding, $L$ is still a line, because $L \cdot (H+[C_{5,1}]) = 1$.

The image is a surface $T_5$ defined by three quadrics $Q_1, Q_2, Q_3$.
For each $Q_i$, let $Z_i$ be the subscheme of $\P^5 \times \P^5$ consisting
of pairs of points $P, P'$ for which $P = P'$ or the line joining
$P$ to $P'$ is on the variety defined by $Q_i = 0$.  The image of
$Z_i$ in $\P^{14}$ under the Pl\" ucker embedding is isomorphic to the
Fano scheme of lines on $Q_i$.  It is easy to check in Magma that the
intersection of the images of the $Z_i$ in $\P^{14}$ is empty, which implies
that there are no lines on $T_5$ even over $\bar \F_5$.
\end{proof}

\begin{remark} In fact, we claim that $Y_5$ contains no lines at all.  Indeed, a calculation similar 
   to the one in the proof of Proposition~\ref{prop:char-5} shows
  that there are also no lines on $Y_5$ disjoint from $C_{5,2}$.  To finish the proof of the claim, 
  it suffices to show that there are no lines that meet both $C_{5,1}$ and $C_{5,2}$. 
  Let $Z_2$ be the closure of the subset of 
  $\P^4_{\F_5} \times \P^4_{\F_5}$ consisting of all pairs of points $(x,y)$ such that $x \neq y$ and 
  the line through $x$ and $y$ lies on the hypersurface defined by $F_5$. 
  Define $Z_3$ similarly for $G_5$.   Since $C_{5,1}$ and $C_{5,2}$ are disjoint,
  the intersection $Z_2 \cap Z_3 \cap (C_{5,1} \times C_{5,2})$ consists of all pairs $(x,y)$ with 
  $x \in C_{5,1}$ and $y\in C_{5,2}$ for which there is a line on $Y_5$ through $x$ and $y$. 
  We then check in Magma that this intersection is
  empty.  This is done in our Magma scripts \cite{script}.

  In principle one can prove directly that $Y_5$ has no lines by checking that
  the images of $Z_2$ and $Z_3$ under the Pl\" ucker embedding in $\P^9$ are
  disjoint, but this calculation is very slow.
\end{remark}

It remains to show that there exists a surface $Y \subset \P^4_\Q$ containing the conics $C_1,C_2$ and lifting the surfaces $Y_3,Y_5$ described above.
Denote by $V_2(C_1 \cup C_2, \Q)$ the vector space of homogeneous forms of degree $2$, with coefficients in $\Q$, vanishing on $C_1 \cup C_2$.
One easily computes that $V_2(C_1 \cup C_2, \Q)$ has dimension $5$; let $q_1, \dotsc, q_5$ be a $\Z$-basis for the free abelian group $V_2(C_1 \cup C_2, \Z)$ of forms having integer coefficients.
One also easily computes that the analogous spaces $V_2(C_{3,1} \cup C_{3,2}, \F_3)$ and $V_2(C_{5,1} \cup C_{5,2}, \F_5)$ both have dimension $5$ over $\F_3$ and $\F_5$ respectively, which implies that the images of $q_1, \dotsc, q_5$ span these spaces.
Now the Chinese remainder theorem guarantees the existence of a quadratic form $F \in V_2(C_1 \cup C_2, \Z)$ satisfying both $F \equiv F_3 \pmod 3$ and $F \equiv F_5 \pmod 5$.
Similarly, the analogous spaces of cubic forms vanishing on $C_1 \cup C_2$ are all of dimension $21$, which implies the existence of a cubic form $G$ satisfying $G \equiv G_3 \pmod 3$ and $G \equiv G_5 \pmod 5$.
Remark~\ref{syzygies} below shows that one can avoid explicit computations to calculate these dimensions.
The surface $Y$ given by $F=G=0$ reduces modulo $3$ and $5$ to $Y_3$ and $Y_5$, respectively.

\begin{remark}\label{syzygies}
Let $k$ be any field.  Let $C_1,C_2 \subset \P^4_k$ be two conics; suppose that the two planes spanned by $C_1,C_2$ respectively meet only in one point $P$, and suppose that $P$ does not lie on either $C_1$ or $C_2$.
Then the space of forms of degree $d$ vanishing on $C_1$ and $C_2$ has dimension
\[
4\binom{d+2}{4}  - 4\binom{d+1}{4} + \binom{d}{4} + 2\binom{d}{2} - 1 =
\tfrac{1}{24}d(d-1)(d^2+11d+46) - 1. 
\]
\end{remark}

\subsection{Finite automorphism group for a family of diagonal surfaces}\label{finite-diagonal}
In this example, for which we summarize the calculations in a Magma script
\cite{script}, we consider a quartic surface $Z$ over a field $k$
of characteristic $0$
defined by
\[
x^4 - y^4 = c(z^4 - w^4).
\]
for some nonzero $c \in k$. Conjecture 2.5 of \cite{vluijk} states that if $k$ is the field $\Q$ of rational 
numbers, then for every nonzero $c \in \Q$, the set of rational points on the associated surface $Z$ is dense.
The $48$ lines on a surface in this
family are defined over the field $k(\zeta_8, \root 4 \of c)$.  We assume
that $k, c$ are such that this is a Galois extension of $k$ of degree $16$:
with $k = \Q$, this is equivalent to $|c|$ being neither a square nor twice
a square.
(This Galois action is case A45 of \cite[Appendix A]{bright}.)
We will prove under this assumption that $\Aut Z$ is finite.  The argument
will proceed by the following steps.

\begin{enumerate}
\item Determination of $\Pic Z$, and comparison to the Picard lattice of a
  K3 surface $V$ with finite automorphism group over $\bar k$.
\item Determination of $\O(\Pic V)$ and $\O(\Pic Z)$ up to finite index.
\item Construction of elements of $R_Z$ and verification that
  $\O(\Pic Z)/R_Z$ is finite.
\end{enumerate}

\begin{remark}
  If the characteristic of $k$ is congruent to $1 \bmod 4$, then $k$ contains
  the $4$th roots of $1$ and this Galois action is not possible.  If it is
  $3 \bmod 4$, then $Z$ is a supersingular surface and $Z_{\bar k}$
  has Picard number $22$, so the argument given here does not apply.  If
  $k$ has characteristic $2$, then of course $Z$ is not a smooth surface.
\end{remark}

Let $L_N$ be the lattice with Gram matrix
\[
N = \left( \begin{array}{cc|c} 0&1\\1&0\\ \hline &&-2I_4 \end{array}\right),
\]
as in Section  \ref{finite-from-overlattice}.
  
\begin{lemma}\label{lem:pic-diagonal}
  There is a sublattice of $L_N$ of index $4$ that is isomorphic to $\Pic Z$.
\end{lemma}

\begin{proof}
  The Picard group of $\bar{Z}$ is generated by the classes of the
  $48$ lines. (See~\cite[Section~6.1]{ssvl} for a history of proofs of this fact.)  
  One easily verifies that the subspace fixed by the action of
  Galois is of rank $6$ and that the intersection form has discriminant $-256$.
  Since the surface has rational points such as $(1,1,0,0)$, the fixed subspace
  of the Picard group of $\bar{Z}$ is in fact the Picard group of $Z$.

  Let $\Lambda$ be the lattice obtained from
  $\Pic Z$ by adjoining $v/2$ for each $v$ in 
  \[
  \{\,\, v  \in \Pic Z \,\, : \,\,  8\mid (v,v) \mbox{ and } 2\mid (v,w) \mbox{ for all } w \in \Pic Z \,\, \}.
  \]
  Then $\Lambda$ contains $\Pic Z$ with index $4$
  and hence has discriminant $-16$.  It is verified in \cite{script}
  that $L_N$ and $\Lambda$ are isomorphic.
\end{proof}

We now proceed to step $2$, determining $\O(L_N)$.
Let $V$ be a K3 surface over $\bar k$ with Picard lattice $L_N$, such as those
constructed in Proposition~\ref{prop:four-a1}.
We choose a basis
$e_1, \dots, e_6$ for $L_N \cong \Pic V$ to consist of the classes
$E, E+O, C_1, C_2, C_3, C_4$, where $E$ is the class of an elliptic fibration
with zero section $O$ and the $C_i$ are the components of the reducible
fibres that do not meet $O$.  In this basis, the Gram matrix of $\Pic V$ is
as above.  Let $G \subset \O(L_N)$ be the image of the symmetric group $\S_4$ by the homomorphism
taking a permutation $\pi$ to the automorphism fixing $E$ and $E+O$ and taking
$C_i$ to $C_{\pi(i)}$.  Let $A_{L_N}$ be the discriminant group of $L_N$ with
its discriminant form, a quadratic form with values in $\Q/2\Z$
\cite[Chapter~14]{huybrechts}.

\begin{lemma}\label{outer-auts} $A_{L_N}$ is generated by the classes of
  $e_i/2$ for $3 \le i \le 6$.  The discriminant form takes
  $\sum_{i=3}^6 a_ie_i/2$ to $-\sum_{i=3}^6 a_i^2/2 \bmod 2$, and the natural map
  $G \to \Aut A_{L_N}$ is an isomorphism.
\end{lemma}

\begin{proof}
  The $e_i/2$ for $3 \le i \le 6$ belong to $L_N^*$ and 
  they generate a group that contains $L_N$ with index $16$.  Since
  $L_N$ has discriminant $16$, they generate $L_N^*$ and their classes generate
  $A_{L_N}$.  The quadratic form takes $\sum_{i=3}^6 a_ie_i/2$ to
  $\sum_{i=3}^6 a_i^2 e_i^2/4 \bmod 2 = -\sum_{i=3}^6 a_i^2/2$.  It follows that
  the only elements of $A_{L_N}$ taken to $-1/2$ by the quadratic form are the
  basis vectors, so every automorphism of $A_{L_N}$ must permute these.
  Conversely it is clear that every permutation of the basis vectors extends
  to an automorphism of $A_{L_N}$ and that this automorphism is the image of a
  unique element of~$G$.
\end{proof}

Recall~\cite[Definition~15.1.1]{huybrechts} that an automorphism of $V$ is \emph{symplectic} if it acts trivially on $\H^0(V,\Omega^2_V)$.

\begin{proposition} 
  The group $\Aut V$ is finite cyclic and the subgroup $\Aut_s V$ of
  symplectic automorphisms of $V$ is trivial.
\end{proposition}

\begin{proof}
By Lemma \ref{nine-curves}, there are exactly nine smooth rational curves on $V$.
Every automorphism of $V$ must preserve the configuration of these curves.
Exactly one of them, the section of the elliptic fibration, intersects four of the others. 
Hence, the only automorphisms of the configuration are those that permute the
reducible fibres. Since the nine curves generate $\Pic V$, we conclude that every 
automorphism of $V$ acts on $\Pic V$ by elements of $G$.
Now, a symplectic automorphism $\alpha$
of a K3 surface $V$ in characteristic $0$ acts
as the identity on the transcendental lattice in $\H^2(V,\Z)$
\cite[Remark 15.1.2]{huybrechts} and hence on its discriminant group,
so in addition it acts trivially on the discriminant group of $\Pic V$
by \cite[Lemma 14.2.5]{huybrechts}.  Since $\alpha$ acts on $\Pic V$ through~$G$, 
combining this with the last lemma
shows that $\alpha$ acts trivially on $\Pic V$.   It follows that
$\alpha$ acts trivially on $\H^2(V,\Z)$; it is therefore the identity
by \cite[Corollary 15.1.6]{huybrechts}.

Since we are in characteristic $0$, there is an exact sequence
\cite[(15.1.3)]{huybrechts}
$$1 \to \Aut_s V \to \Aut V \to \mu_m \to 1$$
for some $m$, where $\Aut_s V$ is the group of symplectic automorphisms
that we have just shown to be trivial.  It follows that $\Aut V$ is
finite cyclic.
\end{proof}

\begin{remark}
  There is always an automorphism of order $2$ of $V$ given by negation of the
  elliptic fibration with respect to the unique section.
  For very general~$V$ the order of $\Aut V$ is $2$, but it might be possible for
  $\Aut V$ to be of order $m = 4, 6$, or $8$ if there is an automorphism that
  permutes the reducible fibres of the fibration nontrivially.  In this case
  the field of definition of $\Pic V$ would contain $\mu_m$, which does not
  happen generically.
\end{remark}

By \cite[Proposition 5.10]{dolgachev}, we know that $\O(L_N)$ is generated by the
automorphisms of the ample cone and the reflections in the classes of rational
curves, together with $-1$;
the automorphisms of the ample cone are just the $G$ of Lemma
\ref{outer-auts}.  Since $G$ is generated by $2$ elements and has $3$ orbits on
the set of rational curves, we have a set of $6$ generators for $\O(L_N)$.

To pass to $\O(\Pic Z)$, we use a simple lemma.

\begin{lemma}\label{commensurable}
  Let $\L$ be a lattice with sublattices $\L_1, \L_2$ with $[\L_1:\L_1 \cap \L_2]$
  finite.  Then the stabilizer of $\L_1 \cap \L_2$ in $\O(\L_1)$ has finite
  index.  If in addition $\L_1, \L_2$ have equal rank, then 
  $\O(\L_1) \cap \O(\L_2)$ has finite index in $\O(\L_1)$ and
  $\O(\L_2)$, where the $\O(\L_i)$ are regarded as subgroups of
  $\O(\L_1 \otimes \Q)$. 
\end{lemma}

\begin{proof}
  There are only finitely many sublattices of $\L_1$ of index
  $[\L_1:\L_1 \cap \L_2]$, so the stabilizer is of finite index.
  If $[\L_1:\L_1 \cap \L_2]$ is finite and $\L_1, \L_2$ have equal rank then
  $[\L_2:\L_1 \cap \L_2]$ is finite as well and $\O(\L_1) \cap \O(\L_2)$ is
  the stabilizer of $\L_1 \cap \L_2$ in $\O(\L_i)$ for $i \in \{1,2\}$.
\end{proof}

\begin{remark} The group generated by $\O(\L_1)$ and $\O(\L_2)$ usually does not
  contain $\O(\L_1)$ with finite index.
\end{remark}

Using Magma \cite{magma} to compute the permutation representation of $\O(L_N)$
on sublattices with quotient $(\Z/2\Z)^2$, we find a set of $30$ generators
for $\O(L_N) \cap \O(\Pic Z)$.  However, some of these are products of others,
so we easily reduce to a set of $9$ generators.
We now perform step $3$, the construction of elements of $R_Z$.

\begin{proposition}\label{prop:finite-orbits}
  There are $14$ Galois orbits of lines and $8$ of conics on $\bar{Z}$ that
  give finite Coxeter groups.
\end{proposition}

\begin{proof}
  We list the conics by observing that their Picard classes are of the form
  $H/2 + M$, where $M \in \frac{1}{2} \Pic \bar{Z}$ is such that $H \cdot M = 0$
  and $M^2 = -3$.  Since $H^\perp \subset \frac{1}{2} \Pic \bar{Z}$  is a
  definite lattice, the set of $M$ with these properties is finite.  Such a
  class is represented by
  a conic if and only if it is integral and not the sum of two classes of lines,
  both of which are easily checked.
  
  We finish the proof by referring to
  Proposition~\ref{rx} for the configurations of curves that
  give finite Coxeter groups.
\end{proof}

\begin{remark} It appears that there are no further generators of $R_Z$.
\end{remark}

\begin{theorem} $\Aut Z$ is finite.
\end{theorem}

\begin{proof}
  We use Proposition \ref{main-result}.  To do so, we search for
  relations among the generators of $\O(L_N) \cap \O(\Pic Z)$ and the $22$ elements
  of $R_Z$ that we have found, simply by checking which products of
  up to $5$ generators of $\O(L_N) \cap \O(\Pic Z)$ and up to $4$ of these together
  with the generators of $R_Z$ give the identity matrix.
  Magma quickly verifies that these relations
  are sufficient to show that
  $\langle \O(L_N) \cap \O(\Pic Z),R_Z\rangle/R_Z$ is finite,
  from which we conclude by Proposition \ref{main-result}
  that $\Aut Z$ is finite.
\end{proof}

\begin{corollary}
  There are finitely
  many classes of geometrically irreducible curves on $Z$ of each genus.
\end{corollary}

\begin{proof}
  This follows by combining the theorem with 
  Corollary \ref{finite-class-orbits}.
\end{proof}

\begin{remark}\label{hard-to-prove-density}
  We conclude that in order to prove that the rational points of $Z$
  are Zariski-dense, we do not have an infinite automorphism group $\Aut Z$ to 
  our disposal. In particular, we will not be able to find an elliptic fibration on $Z$ 
  with a section of infinite order, as translation by such a section would be an 
  automorphism of infinite order.
\end{remark}

\bibliographystyle{abbrv}
\bibliography{k3q}
\end{document}